\documentclass[12pt]{amsart}
\usepackage[utf8]{inputenc}
\usepackage[english]{babel}
\usepackage{subcaption}
\usepackage{enumerate}
\usepackage{needspace}
\usepackage{amsfonts}
\usepackage{comment}
\usepackage{amsmath,amssymb}
\usepackage[makeroom]{cancel}
\usepackage{stmaryrd}
\usepackage{tikz-cd} 
\usepackage{enumerate}
\usepackage{enumitem}
\usepackage{graphicx}
\usepackage{amsthm}
\usepackage{wrapfig}
\usepackage{hyperref}
\usepackage{chngcntr}
\counterwithin{figure}{section}
\newtheorem{theorem}{Theorem}
\newtheorem*{theorem*}{Theorem}
\newtheorem{prop}[theorem]{Proposition}

\newtheorem{coro}{Corollary}[theorem]
\newtheorem{lemm}[theorem]{Lemma}
\theoremstyle{definition}
\newtheorem{defi/}[theorem]{Definition}
\newenvironment{defi}
  {%
   \pushQED{\qed}\begin{defi/}}
  {\popQED\end{defi/}}
\newtheorem*{defi*}{Definition}
\newtheorem*{prop*}{Proposition}
\newtheorem*{lemm*}{Lemma}
\newtheorem{ques}[theorem]{Question}

\newtheorem{cor-definition}[theorem]{Corollary-Definition}
\newtheorem{rema}[theorem]{Remark}

\newtheorem*{conv*}{Convention}
\newtheorem{fact}{Fact}

\numberwithin{theorem}{section}

\newcommand{\clos}[1]{\overline{#1}}
\newcommand{\inte}[1]{\overset{\circ}{#1}}

\numberwithin{equation}{section}
\title{Markovian actions are Anosov-like}
\author{Ioannis Iakovoglou}

\begin{document}
\maketitle
\begin{abstract}
    Motivated by the introduction of Anosov-like actions by Thomas Barthelmé, Steven Frankel, and Kathryn Mann, in this paper we study group actions on the plane that preserve a pair of transverse singular foliations, as well as a strong Markovian family—namely, a collection of rectangles serving as an analogue of a Markov partition for group actions. We prove that any such action is Anosov-like. In particular, the existence of a strong Markovian family constrains the behavior of the invariant bifoliation and endows the action with characteristic features of hyperbolic dynamics, such as the existence of infinitely many elements with fixed points.
\end{abstract}
\section{Introduction}

Following the works of T.Barbot, S.Fenley and L.Mosher, to every pseudo-Anosov flow $\Phi$ (with no circle $1$-prong singularities) on a closed 3-manifold $M$, we can associate a topological plane $\mathcal{P}$ endowed with a pair of transverse singular foliations $(F^s,F^u)$ and a continuous action $\alpha:\pi_1(M)\rightarrow \text{Homeo}(\mathcal{P})$ preserving $F^s$ and $F^u$. Thanks to a theorem by T. Barbot, the data $(\mathcal{P}, F^s,F^u, \alpha)$, considered up to conjugation, completely determine $\Phi$ up to orbital equivalence. The previous results establish a connection from $3$-dimensional pseudo-Anosov flows to group actions on the plane preserving a pair of transverse singular foliations.

Conversely, given a topological plane $\mathcal{P}$, a group $G$ and a continuous action $\rho: G\rightarrow \text{Homeo}(\mathcal{P})$ preserving a pair of transverse singular foliations $(\mathcal{F}^s,\mathcal{F}^u)$, it is natural to ask: 

\begin{ques}\label{q.origin}
   Under which conditions do the data $(\mathcal{P},\mathcal{F}^s,\mathcal{F}^u,\rho)$ arise from a pseudo-Anosov flow on a closed $3$-manifold $M$?
\end{ques}

Motivated by the above question, the authors of \cite{circleatinfinity} defined the notion of \emph{Anosov-like action} on the plane, a necessary condition for a group action to arise from some pseudo-Anosov flow : 

\begin{defi}[Anosov-like action]\footnote{The definition of Anosov-like action introduced in \cite{circleatinfinity} was well adapted to the study of transitive Anosov flows and was later generalized to the definition used here, which coincides with that of \cite{nontransitiveanosovlike}.} \label{d.anosovlike}
Consider $\rho$ a continuous and faithful action on the plane  $\mathcal{P}$ preserving a pair of transverse singular foliations $(\mathcal{F}^s,\mathcal{F}^u)$ with no 1-prong singularities. The action $\rho$ is called \emph{Anosov-like} if it satisfies the following:
\begin{enumerate}\renewcommand{\labelenumi}{(A\arabic{enumi})}
    \item For any non-trivial $g\in G$ and any leaf $l\in \mathcal{F}^{s,u}$, if $\rho(g)(l)=l$, then $\rho(g)$ has a unique fixed point $x$ in $l$. Furthermore, after possibly changing the roles of $\mathcal{F}^s$ and $\mathcal{F}^u$, we have that $$\forall y^s\in \mathcal{F}^s(x)~ \forall y^u\in \mathcal{F}^u(x)~~\rho(g^n)(y^s)\underset{n\rightarrow+\infty}{\longrightarrow} x \text{ and }\rho(g^n)(y^u)\underset{n\rightarrow-\infty}{\longrightarrow} x$$.
    \item  The union of leaves of $\mathcal{F}^s$ (resp. $\mathcal{F}^u$) that are fixed by some non-trivial element of $G$ is dense in $\mathcal{P}$.
    \item Each singular point of $\mathcal{F}^{s,u}$ is fixed by some non-trivial element of $G$.
    \item If $l\in \mathcal{F}^s$ (resp. $l\in \mathcal{F}^u$) is non-separated with some leaf $l'\in \mathcal{F}^s$ (resp. $l'\in \mathcal{F}^u$) in the leaf space of $\mathcal{F}^s$ (resp. $\mathcal{F}^u$), then there exists a non-trivial element $g\in G$ such that $\rho(g)(l)=l$.
    \item There are no totally ideal quadrilaterals in $\mathcal{P}$ (see Definition \ref{d.ideal}).
\end{enumerate}
\end{defi}

Anosov-like actions provide a natural framework for studying pseudo-Anosov flows in dimension $3$. More precisely, the class of Anosov-like actions captures the main hyperbolic features and includes all actions arising from pseudo-Anosov flows in dimension $3$ (see Proposition 2.2 of \cite{nontransitiveanosovlike}). On the other hand, there exist Anosov-like actions that are not associated with any pseudo-Anosov flow (see Example 1 in \cite{circleatinfinity}). Nevertheless, many properties known for actions arising from pseudo-Anosov flows continue to hold in the more general setting of Anosov-like actions (\cite{circleatinfinity,nontransitiveanosovlike}).

In this paper, inspired by the notion of Markovian family that we first defined in \cite{thesisioannis} and motivated by Question \ref{q.origin}, we introduce and study the properties of a new class of Anosov-like actions, defined by the existence of some invariant  ``Markov partition".

\begin{defi}[Markovian family - Informal]\label{d.markovianactioninformal}
Consider $\rho$ a continuous action on the plane  $\mathcal{P}$ preserving a pair of transverse singular foliations $(\mathcal{F}^s,\mathcal{F}^u)$ with no 1-prong singularities. A \emph{Markovian family of $\rho$ in $(\mathcal{P},\mathcal{F}^s,\mathcal{F}^u)$} is a $\rho$-invariant set of rectangles $\mathcal{R}=(R_i)_{i \in I}$ (i.e. topological disks that are trivially bifoliated by $\mathcal{F}^s$ and $\mathcal{F}^u$) satisfying the three following properties: 
\begin{enumerate}
\item $\mathcal{R}$ is finite up to the action by $\rho$ 
\item For every two distinct rectangles $R_i,R_j\in\mathcal{R}$, if $\overset{\circ}{R_i } \cap \overset{\circ}{R_j} \neq \emptyset$, then, up to changing the roles of $R_i$ and $R_j$, $R_i \cap R_j$ is a non-trivial horizontal subrectangle of $R_i$ and a non-trivial vertical subrectangle of $R_j$  
\item For any compact set $U\subset \mathcal{P}$, we have that $U$ can be covered by finitely many rectangles in $\mathcal{R}$ 
\end{enumerate} 
Moreover, the family $(R_i)_{i \in I}$ will be called a \emph{strong Markovian family of $\rho$} if it also satisfies the following: 
\begin{enumerate}
    \item[(4)] If $(R_i)_{i \in \mathbb{Z}}$ is a bi-infinite sequence of rectangles in $\mathcal{R}$ such that for every $i\in\mathbb{Z}$ the intersection $R_{i+1}\cap R_i$ is a non-trivial vertical (resp. horizontal) subrectangle of $R_i$ (resp. $R_{i+1}$), then $\underset{i\in \mathbb{Z} }{\cap}R_i$ contains a unique point. Conversely, every point of $\mathcal{P}$ arises as the intersection of a bi-infinite sequence of rectangles satisfying this property.
\end{enumerate}
\end{defi}
For a more detailed definition of the notion of strong Markovian family, see Definition \ref{d.markovfamily}.

\begin{defi}[Markovian action]\label{d.markovianaction}
    Consider $\rho$ a continuous, faithful action of a countable group $G$ on the plane $\mathcal{P}$ preserving a pair of transverse singular foliations $(\mathcal{F}^s,\mathcal{F}^u)$ with no $1$-prong singularities. We will say that $\rho$ is a \emph{Markovian action} (resp. \emph{strong Markovian action}) if: 
    \begin{enumerate}
        \item $G$ has no torsion
        \item $\rho$ preserves a  Markovian family (resp. strong Markovian family) in $(\mathcal{P},\mathcal{F}^s,\mathcal{F}^u)$ 
    \end{enumerate}   
\end{defi}

The main result in this paper shows that:

\begin{theorem}\label{t.main}
    Every orientation preserving strong Markovian action on the plane is Anosov-like. 
\end{theorem}

Our motivation for introducing the class of strong Markovian actions is twofold. First, in their recent work (see \cite{reconstruct}), T.Barthelmé, S.Fenley and K.Mann answer Question \ref{q.origin} in the case of transversally orientable pseudo-Anosov flows. More specifically, they provide a necessary and sufficient condition for an Anosov-like action to arise from a transversally orientable pseudo-Anosov flow and they also develop techniques allowing to reconstruct transversally orientable pseudo-Anosov flows from Anosov-like actions. Building on the previous techniques, together with F.Béguin, we generalize the previous results and we show in forthcoming work, that an orientation-preserving action on the plane arises from a pseudo-Anosov flow if and only if it is a strong Markovian action. Hence, preserving a strong Markovian family is the essential property separating general group actions on the plane preserving some bifoliation from actions that arise from pseudo-Anosov flows in dimension 3. 

Second, because of the combinatorial nature of their definition, strong Markovian actions on the plane form a subclass of Anosov-like actions that is more amenable to a complete classification than general Anosov-like actions. In fact, in our previous work in \cite{monstre}\footnote{In \cite{monstre} we defined a strong Markovian action as an Anosov-like action satisfying the three properties of Definition \ref{d.markovianactioninformal}. However, by Theorem \ref{t.main} the condition Anosov-like can be dropped in the previous definition.}, we reduced the study of a strong Markovian action to the study of a finite combinatorial object, thus forming a path towards a complete classification of strong Markovian actions and thus also of pseudo-Anosov flows in dimension 3.

We now turn to the precise definitions and the proof of Theorem \ref{t.main}.
\section{Preliminaries}
\subsection{On the pairs of singular foliations on the plane}\label{s.singfolidefi}

Consider $\mathcal{F}^2_h,\mathcal{F}^2_v$ the foliations by horizontal and vertical lines on $\mathbb{R}^2=\mathbb{C}$, $\mathcal{D}_2$ the euclidean open square $\{z\in\mathbb{C}||\text{Re}(z)|<1 \text{, } |\text{Im}(z)|<1\}$ and $\pi_p: \mathbb{C}\rightarrow \mathbb{C}$ the map associating $z$ to $z^p$, where $p\in \mathbb{N}^*$. Let $\mathcal{D}_1:= \pi_2(\mathcal{D}_2)$ and $\mathcal{D}_p:= \pi_p^{-1}(\mathcal{D}_1)$ for any $p\geq 3$. 

The set of images of the leaves of $\mathcal{F}^2_h,\mathcal{F}^2_v$ by $\pi_2$ define two singular foliations $\mathcal{F}^1_h,\mathcal{F}^1_v$ (see Figure \ref{f.prongsingularities}). Similarly, for every $p\geq 3$, the set of pre-images of the leaves $\mathcal{F}^1_h,\mathcal{F}^1_v$ by $\pi_p$, define two singular foliations, say $\mathcal{F}^p_h,\mathcal{F}^p_v$ (see Figure \ref{f.prongsingularities}).  

Throughout this paper, we will call a foliation on a surface \emph{singular} if it only admits prong-type singularities. More specifically,  
\begin{defi}\label{d.singularfolisurfaces}
    Consider $\Sigma$ a surface with empty boundary. We will say that $\mathcal{F}$ is \emph{a singular foliation of $\Sigma$} if it is a partition  of $\Sigma$ into \emph{regular} and \emph{singular leaves} such that: 
    \begin{enumerate}
    \item $\mathcal{F}$ contains a countable number (possibly zero) of singular leaves $(L_i)_{i\in I}$.
    \vspace{0.1cm}
    
    \item for every $i\in I$ we have that $L_i$ contains a unique point $\sigma_i$ for which there exist $U_{\sigma_i}\subset \Sigma$ a neighborhood of $\sigma_i$, $p_i\in \mathbb{N}\setminus\{0,2\}$ and $h:U_{\sigma_i}\rightarrow \mathcal{D}_{p_i}$ a homeomorphism verifying $$h(\sigma_i)=(0) \text{ and } h(\mathcal{F}\cap U_{\sigma_i})=\mathcal{F}^{p_i}_h\cap \mathcal{D}_{p_i}$$

    \item $\{\sigma_i|i\in I\}$ is a discrete subset of $\Sigma$ 
    \vspace{0.1cm}
    
    \item the restriction of $\mathcal{F}$ on $\Sigma-\{\sigma_i|i\in I\}$ defines a $C^0$ line foliation. 
    \end{enumerate}
    We will call $\sigma_i$ a \emph{$p_i$-prong singularity of $\mathcal{F}$} or more simply \emph{a singularity of $\mathcal{F}$} and we will denote by $\text{Sing}(\mathcal{F})$ the set of singularities of $\mathcal{F}$.  Any non-singular point of $\Sigma$ will be called a \emph{regular point of $\mathcal{F}$}.  

   \end{defi}
 \begin{figure}[h!]

  \begin{minipage}[ht]{0.4\textwidth}
    \centering 
     \vspace{0.8cm}
    \hspace{-0.5cm}
    \includegraphics[width=0.9\textwidth]{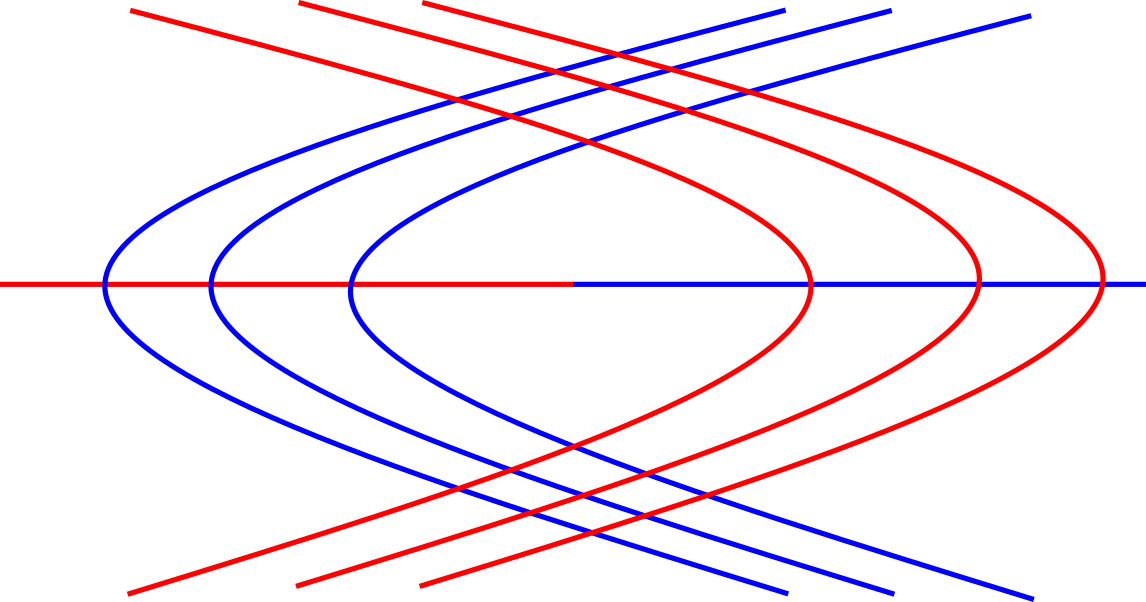}
  \hspace{-1cm}

  \end{minipage}
 \begin{minipage}[ht]{0.4\textwidth}
 \centering
    \includegraphics[width=0.8\textwidth]{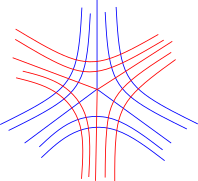}
    \hspace{-1cm}
    
  \end{minipage}
\caption{On the left, the foliations $\mathcal{F}^1_h,\mathcal{F}^1_v$. On the right, $\mathcal{F}^3_h,\mathcal{F}^3_v$.}
\label{f.prongsingularities}
  \end{figure} 

Similarly to transverse regular line foliations, we define transverse singular foliations as follows:

\begin{defi}\label{d.transversefolisurfaces}
   Consider $\mathcal{F}, \mathcal{G}$ two singular foliations on a surface $\Sigma$ with empty boundary. We will say that $\mathcal{F}, \mathcal{G}$ are \emph{transverse} if for every point $x\in \Sigma$ there exists $U_x$ a neighborhood of $x$ in $\Sigma$, $p\in \mathbb{N}^*$ and $h:U_x\rightarrow \mathcal{D}_p$ a homeomorphism such that $h(x)=0$, $$h(\mathcal{F}\cap U_x)=\mathcal{F}_h^p\cap \mathcal{D}_p \text{  and  } h(\mathcal{G}\cap U_x)=\mathcal{F}_v^p\cap \mathcal{D}_p$$ 
\end{defi}
An alternative way for defining a pair of transverse singular foliations on $\Sigma$ consists in finding an atlas on $\Sigma$ made of \emph{polygonal} charts that respect the pair of foliations. 
\begin{defi}
Consider a singular foliation $\mathcal{F}$ on a surface $\Sigma$ without boundary.

\begin{itemize}
    \item Any segment contained in a leaf of $\mathcal{F}$ will be called an $\mathcal{F}$-\emph{segment}, or simply a \emph{segment of $\mathcal{F}$}.
    
    \item For any $x\in\Sigma$, we will denote by $\mathcal{F}(x)$ the leaf of $\mathcal{F}$ containing $x$.
    
    \item For any $x\in\Sigma$, the closure (with respect to the leaf topology) of a connected component of $\mathcal{F}(x)-\{x\}$ will be called an \emph{$\mathcal{F}$-separatrix} of $x$.
\end{itemize}
\end{defi}
\begin{defi}[Polygons and standard polygons]\label{d.standardpolygon}
 Recall that for any $p\in \mathbb{N}^*$ the unique singular point inside the bifoliated plane $(\mathbb{R}^2,\mathcal{F}_{h}^p,\mathcal{F}_{v}^p)$ lies at the origin and is a $p$-prong singularity. 
 
 A closed topological disk $P$ in $(\mathbb{R}^2,\mathcal{F}_{h}^p,\mathcal{F}_{v}^p)$ will be called a \emph{polygon} if it is bounded by segments of $\mathcal{F}_{h}^p$ and $\mathcal{F}_{v}^p$ that alternate between them (see Figure \ref{f.polygons}). We will say that the polygon $P$ is a \emph{standard polygon} if one of the two following conditions is verified:
    \begin{itemize}
        \item $P$ does not contain the origin in its interior and its boundary consists of $2$ segments of $\mathcal{F}_{h}^p$ and $2$ segments of $\mathcal{F}_{v}^p$. In this case, $P$ is trivially bifoliated and is called a \emph{rectangle}.
        \item $P$ contains the origin in its interior and its boundary consists of $p$ segments of $\mathcal{F}_{h}^p$ and $p$ segments of $\mathcal{F}_{v}^p$. In this case, we will call $P$ a \emph{standard $2p$-gon}.
    \end{itemize}

Consider now $\mathcal{F}, \mathcal{G}$ two transverse singular foliations on a surface $\Sigma$ with empty boundary. A closed topological disk $P$ in $(\Sigma, \mathcal{F},\mathcal{G})$ will be called a \emph{polygon} (resp. \emph{standard polygon}, \emph{rectangle}, \emph{standard $2p$-gon}) if there exists $U\subset \Sigma$ a neighborhood of $P$, $p\in \mathbb{N}^*$, $V$ a neighborhood of the origin in $(\mathbb{R}^2,\mathcal{F}_{h}^p,\mathcal{F}_{v}^p)$ and a homeomorphism $h:U\rightarrow V$ such that: 
\begin{itemize}
    \item $h(\mathcal{F}\cap U)=\mathcal{F}_{h}^p\cap V$ and $h(\mathcal{G}\cap U)=\mathcal{F}_{v}^p\cap V$
    \item $h(P)$ is a polygon (resp. standard polygon, rectangle, standard $2p$-gon) in $(\mathbb{R}^2, \mathcal{F}_{h}^p,\mathcal{F}_{v}^p)$
\end{itemize}
\end{defi}

As an immediate consequence of Definitions \ref{d.transversefolisurfaces} and \ref{d.standardpolygon}, we have that two singular foliations $\mathcal{F}, \mathcal{G}$ on a surface $\Sigma$ with empty boundary are transverse if and only if every point of $\Sigma$ is contained in the interior of a standard polygon in $(\Sigma, \mathcal{F}, \mathcal{G})$. 

\begin{figure}[h]

  \begin{minipage}[ht]{0.4\textwidth}
    \centering 
     \vspace{0.8cm}
    \hspace{-0.5cm}
    \includegraphics[width=0.28\textwidth]{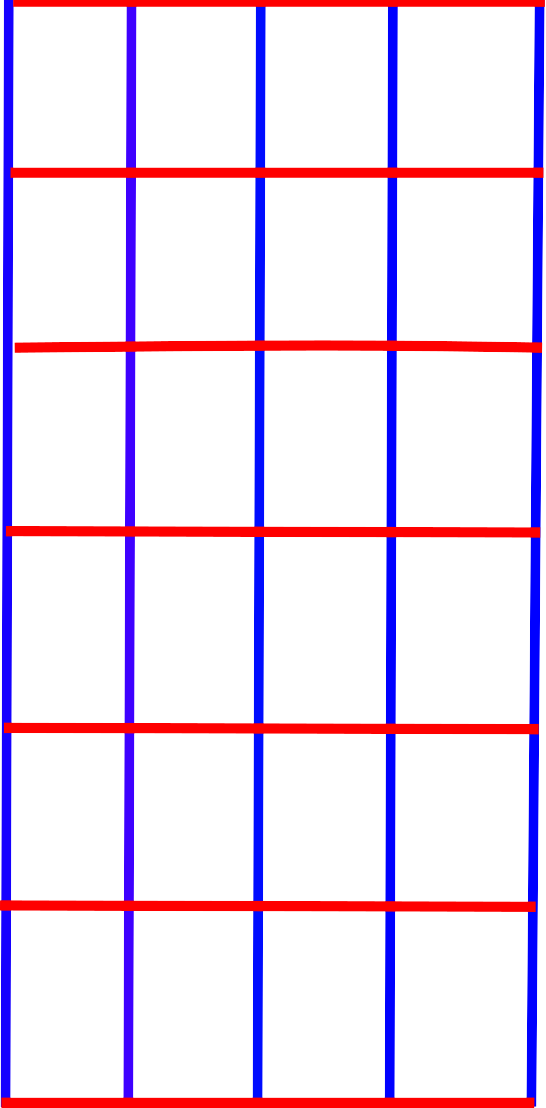}
  \hspace{-1cm}
    \caption*{\quad (a)}
    
  \end{minipage}
 \begin{minipage}[ht]{0.4\textwidth}
 \centering
 \vspace{0.9cm}
    \includegraphics[width=0.55\textwidth]{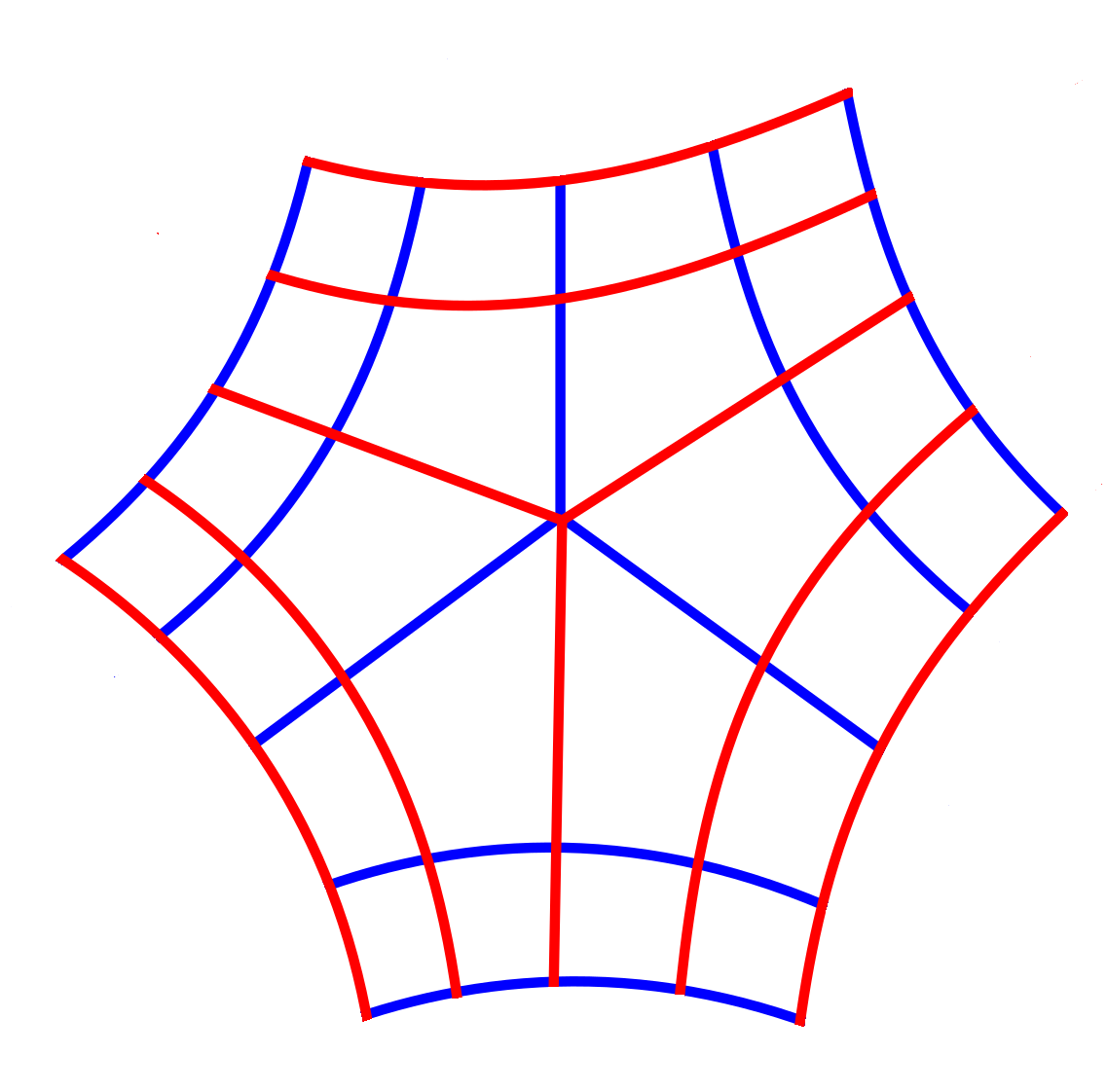}
    
    \caption*{(b)}
    
  \end{minipage}
\caption{Two standard polygons in $\Sigma$.}
\label{f.polygons}
  \end{figure}

In this paper, we will be mainly interested in pairs of transverse singular foliations on the plane with no $1$-prong singularities. The previous pairs of singular foliations are known to satisfy many topological properties : 

\begin{prop}\label{p.propertiesoffoliformarkovianactions}
Consider $\mathcal{F}$ and $\mathcal{G}$ two transverse singular foliations on a topological plane $\mathcal{P}$ with no $1$-prong singularities. Let $\text{Sing}:= \text{Sing}(\mathcal{F})=\text{Sing}(\mathcal{G})$. We have that 
    \begin{enumerate}
        \item The foliations $\mathcal{F}$ and $\mathcal{G}$ do not admit compact leaves. Moreover, if $\sigma\in \text{Sing}$, then $\sigma$ does not admit compact $\mathcal{F}$-separatrices or compact $\mathcal{G}$-separatrices.
        \item Each leaf of $\mathcal{F}$ or $\mathcal{G}$ is closed and properly embedded in $\mathcal{P}$. More specifically, a regular leaf of $\mathcal{F}$ or $\mathcal{G}$ is a properly embedded line in  $\mathcal{P}$. Similarly, if $\sigma\in Sing$ is a $p$-prong singularity, then $\sigma$ admits $p$ distinct $\mathcal{F}$-separatrices and $p$ distinct $\mathcal{G}$-separatrices, each one of which is a properly embedded closed half-line in $\mathcal{P}$.
        \item Any regular leaf of $\mathcal{F}$ or $\mathcal{G}$ separates $\mathcal{P}$ into 2 simply connected components. Similarly, if $\sigma\in Sing$ is a $p$-prong singularity of $\mathcal{F}$ and $\mathcal{G}$, then $$\mathcal{P}-\mathcal{F}(\sigma) ~~\big(\text{resp. }\mathcal{P}-\mathcal{G}(\sigma),~ \mathcal{P}-(\mathcal{F}(\sigma)\cup \mathcal{G}(\sigma))\big)$$ consists of $p$ (resp. $p$, $2p$) simply connected components, the closure of each one of which is homeomorphic to a closed half-plane. 
        \item Every leaf of  $\mathcal{F}$ intersects at most once any leaf of $\mathcal{G}$.
        \item Any leaf of  $\mathcal{F}$  (resp. $\mathcal{G}$) intersects a standard polygon in $(\mathcal{P}, \mathcal{F}, \mathcal{G})$ at most along a unique $\mathcal{F}$-segment (resp. $\mathcal{G}$-segment).

    \end{enumerate}
\end{prop}
The above proposition is classical; a proof can be found, for instance, in \cite{monstre} (see Proposition 2.63). We conclude this section with the following definition, based on the above proposition. 

\begin{defi}
Let $\mathcal{F}$ and $\mathcal{G}$ be two transverse singular foliations on a topological plane $\mathcal{P}$ with no $1$-prong singularities.

\begin{itemize}
    \item For any $x\in \mathcal{P}$, the closure of a connected component of $\mathcal{P}-(\mathcal{F}(x)\cup \mathcal{G}(x))$ is called a \emph{quadrant} of $x$ in $(\mathcal{P},\mathcal{F},\mathcal{G})$. Moreover, a neighborhood of $x$ inside one of its quadrants $Q$ is called a \emph{germ of $x$ in $Q$}.
    \item For any $x\in \mathcal{P}$, the boundary of any connected component of $\mathcal{P}-\mathcal{F}(x)$ will be called a \emph{face} of $\mathcal{F}(x)$.
    \item Given a standard polygon $P$ in $(\mathcal{P},\mathcal{F},\mathcal{G})$ and $x\in P$, the connected component of $\mathcal{F}(x)\cap P$ (resp. $\mathcal{G}(x)\cap P$) containing $x$ is called the \emph{$\mathcal{F}$-leaf} (resp. \emph{$\mathcal{G}$-leaf}) of $x$ in $P$.

\end{itemize}
\end{defi}
\subsection{Perfect fits and ideal quadrilaterals}
In this small section, we define two specific leaf configurations that can appear inside a bifoliated plane. Consider $\mathcal{P}$ a topological plane endowed with a pair $(\mathcal{F}^s,\mathcal{F}^u)$ of transverse singular foliations with no $1$-prong singularities. 

We define a \emph{perfect fit} as a pair of leaves in $(\mathcal{F}^s,\mathcal{F}^u)$ that do not intersect, but almost intersect (see Figure \ref{f.leafconf}a):
\begin{defi}\label{d.perfectfit}
    Fix $L^s\in \mathcal{F}^s$, $L^u\in \mathcal{F}^u$ such that $L^s\cap L^u=\emptyset$. Denote by $D_s$ (resp. $D_u$) the connected component of $\mathcal{P}-L^s$ (resp. $\mathcal{P}-L^u$) containing $L^u$ (resp. $L^s$). We will say that $L^s$ and $L^u$ form a \emph{perfect fit} if: 
    \begin{enumerate}
        \item for every small $\mathcal{F}^u$-segment $I^u$ intersecting $L^s$ and every $x\in I^u\cap D_s\cap D_u$ we have that $\mathcal{F}^s(x)\cap L^u\neq \emptyset$
        \item for every small $\mathcal{F}^s$-segment $I^s$ intersecting $L^u$ and every $x\in I^s\cap D_s\cap D_u$ we have that $\mathcal{F}^u(x)\cap L^s\neq \emptyset$ 
    \end{enumerate}
\end{defi}
\begin{defi}\label{d.ideal}
    A \emph{totally ideal quadrilateral} (see Figure \ref{f.leafconf}b) is an open set $U$ in $\mathcal{P}$ satisfying the following properties: 
\begin{itemize}
    \item there exist four leaves $l_1,l_3\in \mathcal{F}^s$ and $l_2,l_4\in \mathcal{F}^u$ such that $U$ is bounded by the disjoint union of four faces contained respectively in $l_1,l_2,l_3,l_4$ 
    \vspace{0.1cm}
    
    \item $l_i$ makes a perfect fit with $l_{i+1}$ (the indexes are considered modulo 4)
      \vspace{0.1cm}
      
    \item for any leaf $l\in \mathcal{F}^s$ (resp. $l\in \mathcal{F}^u$) that intersects $U$ we have that $l$ intersects $l_2$ (resp. $l_1$) if and only if it intersects $l_4$ (resp. $l_3$)
\end{itemize}
\end{defi}
\begin{figure}[h]

  \begin{minipage}[ht]{0.4\textwidth}
    \centering 
     \vspace{0.8cm}
    \hspace{-0.7cm}
    \includegraphics[width=0.6\textwidth]{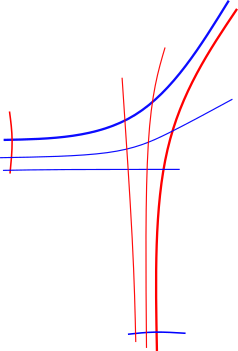}

    \caption*{\quad (a)}
    
  \end{minipage}
 \begin{minipage}[ht]{0.4\textwidth}
 \centering
 \vspace{1.9cm}
    \includegraphics[width=0.56\textwidth]{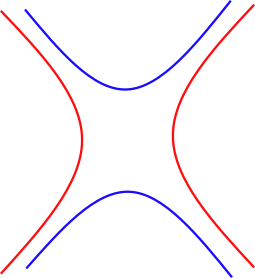}
 \vspace{0.65cm}    
    \caption*{(b)}
    
  \end{minipage}
\caption{On the right a perfect fit and on the left a totally ideal quadrilateral.}
\label{f.leafconf}
  \end{figure}

\subsection{Strong Markovian families}
Consider $\mathcal{P}$ a topological plane endowed with a pair $(\mathcal{F}^s,\mathcal{F}^u)$ of transverse singular foliations  with no $1$-prong singularities. Denote by $\rho:G \rightarrow \text{Homeo}(\mathcal{P})$ a faithful and orientation preserving $C^0$ action of a countable group $G$ on the plane preserving both $ \mathcal{F}^s$ and $\mathcal{F}^u$.

\begin{defi}\label{d.subrectangleplane}
Let $R$ be a rectangle in $\mathcal{P}$ (see Definition \ref{d.standardpolygon}). We will call the union of the two $\mathcal{F}^s$-segments in the boundary of $R$ the \emph{$\mathcal{F}^s$-boundary} of $R$ and we will denote it by $\partial^sR$. We similarly define $\partial^uR$, the \emph{$\mathcal{F}^u$-boundary} of $R$. Any connected component of $\partial^sR$  (resp. $\partial ^uR$) will be called an \emph{$\mathcal{F}^s$-boundary} (resp. \emph{$\mathcal{F}^u$-boundary}) \emph{component} of $R$.

Moreover, any subrectangle $R'$ of $R$ such that $\partial^sR'\subset \partial^sR $ will be called a \emph{vertical subrectangle} of $R$. If furthermore $R' \neq R$, then we will call $R'$ a \emph{non-trivial} vertical subrectangle of $R$. We define similarly \emph{horizontal subrectangles} and \emph{non-trivial horizontal subrectangles} of $R$.

\end{defi}

Consider $\Sigma$ any closed surface, $f$ a pseudo-Anosov homeomorphism on $\Sigma$ and $\mathcal{M}$ a Markov partition of $f$. Denote by $\widetilde{\mathcal{M}}_{markov}$ the lifts of all the rectangles in $\underset{n\in\mathbb{Z}}{\cup}f^n(\mathcal{M})$ on $\widetilde{\Sigma}\approx \mathbb{R}^2$, the universal cover of $\Sigma$ (see Figure \ref{f.famillemarkov}). The family $\widetilde{\mathcal{M}}_{markov}$ is invariant by the action of the deck transformation group of the cover $\widetilde{\Sigma}\rightarrow \Sigma$ and also by the action of any lift of $f$ on $\widetilde{\Sigma}$.

We will not define these notions here, but this example provides some intuition for the definition of a Markovian family, whose axioms are designed to capture the main properties of $\widetilde{\mathcal{M}}_{markov}$. 

\begin{defi}\label{d.markovfamily}
A \emph{Markovian family of $\rho$ in $(\mathcal{P},\mathcal{F}^s,\mathcal{F}^u)$} is a $\rho$-invariant set of rectangles $\mathcal{R}=(R_i)_{i \in I}$ such that: 
\begin{enumerate}
\item (Finiteness axiom) $\mathcal{R}$ is finite up the action by $\rho$ 
\item (Markovian intersection axiom) For every two distinct rectangles $R_i,R_j\in\mathcal{R}$, if $\overset{\circ}{R_i } \cap \overset{\circ}{R_j} \neq \emptyset$, then, up to changing the roles of $R_i$ and $R_j$, $R_i \cap R_j$ is a non-trivial horizontal subrectangle of $R_i$ and a non-trivial vertical subrectangle of $R_j$ 
\item (Finite return time axiom) For any point $x\in \mathcal{P}$ and any small germ of $x$, say $U$, inside one of its quadrants, there exists $R\in \mathcal{R}$ such that $U\subset R$. 
\end{enumerate} 
Furthermore, we will call $\mathcal{R}$ a \emph{strong Markovian family of $\rho$} if it satisfies the finiteness and the Markovian intersection axioms, and also the two following conditions: 
\begin{enumerate}
    \item[(3$'$)] (Strong finite return time axiom)  For any point $x\in \mathcal{P}$ and any small germ of $x$, say $U$, inside one of its quadrants, there exist $R, R_h, R_v\in \mathcal{R}$ such that $R_h\cap R$ is a non-trivial horizontal subrectangle of $R$, $R_v\cap R$ is a non-trivial vertical subrectangle of $R$ and $U\subset R\cap R_h\cap R_v$. 
    \item[(4)](Expansivity axiom) If $(R_n)_{n\in \mathbb{N}}$ is a sequence of rectangles in $\mathcal{R}$ such that $R_{i+1}\cap R_i$ is a non-trivial vertical (resp. horizontal) subrectangle of $R_i$, then $\underset{i\geq 0 }{\cap}R_i$ is equal to the $\mathcal{F}^u$-leaf (resp. $\mathcal{F}^s$-leaf) of some point of $ R_0$
\end{enumerate}
\end{defi}
\begin{figure}
    \centering
    \includegraphics[scale=0.3]{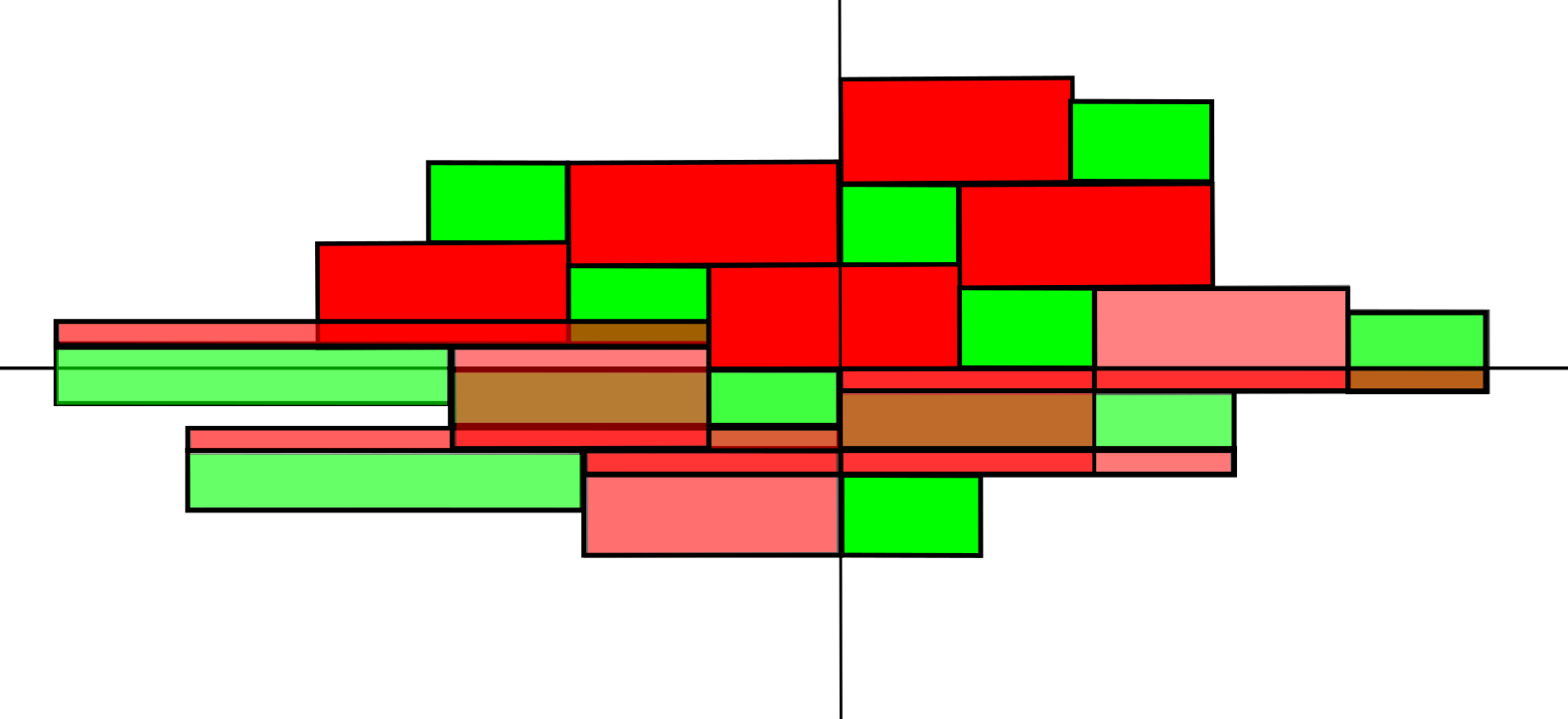}
    \caption{An example of a Markovian family on the plane.}
    \label{f.famillemarkov}
\end{figure}
\subsection{Predecessors and successors for strong Markovian families}
Following the notations of the previous section, consider $\mathcal{R}=(R_i)_{i\in I}$ a strong Markovian family in $(\mathcal{P}, \mathcal{F}^s,\mathcal{F}^u)$ that is invariant by $\rho$. 

Building on the example described earlier, recall that $\mathcal{M}$ was defined as a Markov partition of the homeomorphism $f$ on $\Sigma$; hence, any two rectangles in $\mathcal{M}$ have disjoint interiors. Thanks to the previous fact, for any two rectangles of $\widetilde{\mathcal{M}}_{markov}$ whose interiors intersect, there exist $k\neq l$ such that the projection of the first rectangle on $\Sigma$ belongs in $f^k(\mathcal{M})$ and the projection of the second in $f^l(\mathcal{M})$. This allows us to partially order the rectangles in $\widetilde{\mathcal{M}}_{markov}$. A similar result holds for general Markovian families: 

\begin{lemm}\label{l.existenceofpredecessors}
For any rectangle $R \in \mathcal{R}$ there exists a unique finite collection of rectangles $R_1,...,R_n \in \mathcal{R}$ intersecting $R$ along non-trivial vertical subrectangles and such that:
\begin{enumerate}
\item $R_1,...,R_n$ are maximal for the previous property: any $R' \in \mathcal{R}$ intersecting $R$ along a non-trivial vertical subrectangle satisfies $R' \cap R \subseteq R_i \cap R$ for some $i \in \llbracket 1, n \rrbracket$ 
\item $R_1,...,R_n$ have disjoint interiors 
\item The $R_1,...,R_n$ cover $R$: $\overset{n}{\underset{i=1}{\cup}} R_i \cap R = R$ 
\end{enumerate}
\end{lemm}
\noindent Naturally, the analogue of the previous lemma for horizontal subrectangles is also true. 

Lemma \ref{l.existenceofpredecessors} was proven for strong Markovian families preserved by Anosov-like strong Markovian actions in \cite{monstre} (see Lemma 4.2). However, the proof does not rely on either the Anosov-like or the strong Markovian nature of the action, and therefore carries over verbatim to our setting.

\begin{defi}\label{d.successor}
For any $R\in \mathcal{R}$, we will say that $R'$ is a \emph{predecessor} (resp. \emph{successor}) of $R$ if $R'\cap R$ is a non-trivial vertical (resp. horizontal) subrectangle of $R$ and $R'$ is maximal for this property in the sense of the previous lemma. 

We will say that $R'$ is a predecessor of \emph{$2$-nd generation} of $R$, if $R'$ is a predecessor of a predecessor of $R$. We define similarly a predecessor (resp. successor) of \emph{$n$-th generation} for any $n \geq 3$. By convention, a predecessor (resp. successor) of \emph{$1$-st generation} of $R$, is the same thing as a predecessor (resp. successor) of $R$. Also, the unique predecessor (resp. successor) of $R$ of \emph{generation $0$} is $R$ itself. 

\end{defi}

\begin{lemm}\label{l.npredecessor}
Let $R, R' \in \mathcal{R}$ be such that $R' \cap R$ is a non-trivial vertical subrectangle of $R$. Then there exists a unique $n \in \mathbb{N}$ such that $R'$ is a predecessor of $n$-th generation of $R$. Furthermore, $R$ is a successor of $n$-th generation of $R'$ for the same integer $n$. Finally, there exists a unique sequence in $\mathcal{R}$, say $R_0=R, R_1,\ldots,R_n=R'$, such that $R_{i+1}$ is a predecessor of $R_i$ for every $i\in \llbracket 0,n-1\rrbracket$. 

\end{lemm}
\begin{lemm}\label{l.npredecessorsdisjoint}
Take $R \in \mathcal{R}$ and $N\in \mathbb{N}$. We have that the predecessors (resp. successors) of $R$ of generation $N$ have disjoint interiors. 

\end{lemm}

Once again, Lemmas \ref{l.npredecessor} and \ref{l.npredecessorsdisjoint} were proven for strong Markovian families preserved by Anosov-like strong Markovian actions in \cite{monstre} (see Remark 4.4, Lemmas 4.5 and 4.6). However, the proofs given in \cite{monstre} apply in the exact same way in our setting. 
%\begin{proof}
%Take $R,R'\in \mathcal{R}$ such that $R'\cap R$ is a non-trivial vertical subrectangle of $R$. By the definition of a predecessor, if $R'$ is not one of the predecessors of $R$, then there exists $R_1$ a predecessor of $R$ containing $R'\cap R$. Notice that since the predecessors of $R$ have disjoint interiors, $R_1$ is the unique predecessor of $R$ containing $R'\cap R$. Furthermore, by the Markovian intersection axiom, $R'$ intersects $R_1$ along a non-trivial vertical subrectangle. Again, if $R'$ is not one of the predecessors of $R_1$, then $R'\cap R_1$ is contained in a unique predecessor of $R_1$, say $R_2$. We construct in this way a sequence $R_0=R,R_1,...,R_i,...$ such that for every $i$, the rectangle  $R_{i+1}$ is the predecessor of $R_i$ containing $R'\cap R_i$. If there exists no $i$ such that $R_i\cap R=R'\cap R$, the previous sequence is infinite and by the expansivity property, $\overset{+\infty}{\underset{k=0}{\cap}} R_{k}$
%is an $\mathcal{F}^u$-leaf of $R$ containing $R'\cap R$, which is impossible. We conclude that there exists $n$ such that $R_n\cap R=R'\cap R$. The Markovian intersection axiom implies in this case that $R_n=R'$ (we are utilizing here the fact that the intersection between two rectangles of $\mathcal{R}$ was assumed to be \emph{non-trivial}, see Definition \ref{d.subrectangleplane}), which gives us the desired result.
%\end{proof}

\begin{figure}[h]
    \centering
    \includegraphics[scale=0.45]{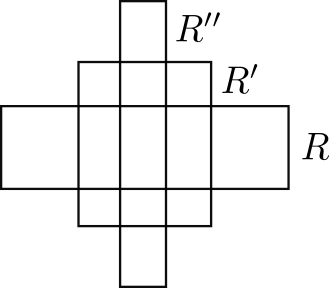}
    \caption{In the above figure, $R''$ is a predecessor of $R'$ of some generation and $R'$ is a predecessor of $R$ of some generation.}
    \label{f.predecessorsinchain}
\end{figure}

We finish this section by remarking that the relation ``being a predecessor of some generation" defines a partial order on $\mathcal{R}$: 

\begin{rema}\label{r.precedentsuivant}
Take $R, R', R'', R''' \in \mathcal{R}$ such that $R'$ is a predecessor of $n$-th generation of $R$, $R''$ is a predecessor of $k$-th generation of $R'$ and $R'''$ is a successor of $l$-th generation of $R'$ (see Figure \ref{f.predecessorsinchain}). Thanks to Lemmas \ref{l.npredecessor} and \ref{l.npredecessorsdisjoint}, we have that: 
\begin{enumerate}  
    \item $R''$ is a predecessor of $(n+k)$-generation of $R$
    \item If $\inte{R}\cap \inte{R'''}\neq \emptyset$ and $l<n$, then $R'''$ is a predecessor of $R$ of generation $n-l$
    \item If $\inte{R}\cap \inte{R'''}\neq \emptyset$ and $l=n$, then $R'''=R$ 
    \item If $\inte{R}\cap \inte{R'''}\neq \emptyset$ and $l>n$, then $R'''$ is a successor of $R$ of generation $l-n$
     \item For every $g\in G$ we have that $\rho(g)(R')$ is a predecessor of $n$-th generation of $\rho(g)(R)$
\end{enumerate}

\end{rema}

\section{On the properties of strong Markovian actions}
From this point on, we begin the proof of Theorem \ref{t.main}. In the following sections, we verify successively that $\rho$ satisfies the axioms of Definition \ref{d.anosovlike}. 

Fix $\mathcal{P}$ a plane endowed with a pair $(\mathcal{F}^s,\mathcal{F}^u)$ of transverse singular foliations  with no $1$-prong singularities. Let $G$ be a countable group with no torsion and let $\rho:G \rightarrow \text{Homeo}_+(\mathcal{P})$ be an orientation preserving strong Markovian action that preserves both the bifoliation  $(\mathcal{F}^s,\mathcal{F}^u)$ and a strong Markovian family $\mathcal{R}:=(R_i)_{i \in I}$. Denote by $e$ the trivial element in $G$.

\subsection{On the periodic points and leaves of the action $\rho$}

Our goal in this section consists in showing that $\rho$ satisfies Axioms (A2) and (A3) of Definition \ref{d.anosovlike}. Once these results have been proved, we will show that, for every $g\in G-\{e\}$, the set of fixed points of $\rho(g)$ in $\mathcal{P}$ is discrete. This well-known property of Anosov-like actions will be used repeatedly throughout the remainder of the paper.

\begin{defi}
A leaf $l$ of $\mathcal{F}^s$ or $\mathcal{F}^u$ is said to be \emph{periodic} if there exists $g \in G-\{e\}$ such that $\rho(g)(l) = l$. We define in the same way the notion of a \emph{periodic point} in $\mathcal{P}$.
\end{defi}
\begin{defi}\label{d.boundaryleaf}
   We will call a leaf $l$ in $\mathcal{F}^s$ (resp. $\mathcal{F}^u$) a \emph{boundary leaf for $\mathcal{R}$} if there exists  $R\in \mathcal{R}$ such that $l$ contains a connected component of $\partial ^sR$ (resp. $\partial ^uR$). 
\end{defi}

According to the following lemma, the action $\rho$ admits a plethora of periodic leaves in both $\mathcal{F}^s$ and $\mathcal{F}^u$. 
\begin{lemm} \label{l.boundaryleavesareperiodic}
Any boundary leaf for $\mathcal{R}$ in $\mathcal{F}^s$ or $\mathcal{F}^u$ is periodic for $\rho$. 
\end{lemm}
\begin{proof}
Fix $R\in \mathcal{R}$, $s$ a connected component of $\partial^s R$ and $L^s$ the leaf in $\mathcal{F}^s$ containing $s$. By repeatedly applying the strong finite return time axiom together with the Markovian intersection axiom (see Definition \ref{d.markovfamily}), there exists $R_0,R_1,R_2,...$ a sequence of rectangles in $\mathcal{R}$ such that $R_0=R$ and also such that $R_{i+1}\cap R_i$ is a horizontal subrectangle of $R_i$ containing $s$ for every $i\in\mathbb{N}$. 

Thanks to Item (5) of Proposition \ref{p.propertiesoffoliformarkovianactions} and to the Markovian intersection axiom, we have that $L^s\cap R_i$ is a connected component of $\partial ^s R_i$ for every $i\in \mathbb{N}$. Next, thanks to the finiteness axiom, after possibly passing to a subsequence, we can assume that $R_0,R_1,....$ belongs in a unique $\rho$-orbit of rectangles in $\mathcal{R}$; in particular, for every $i\in \mathbb{N}$ there exists $g_i\in G-\{e\}$ such that $\rho(g_i)(R_0)=R_i$. Passing to a further subsequence if necessary, we may assume that $\rho(g_i)$ sends the $\mathcal{F}^s$-boundary component of $R_0$ containing $s$ to the $\mathcal{F}^s$-boundary component of $R_i$ containing $s$. Since $\rho$ preserves $\mathcal{F}^s$, this implies that $\rho(g_i)(L^s)=L^s$ and, thus that the leaf $L^s$ is periodic for $\rho$. 

\end{proof}
The following lemma together with Lemma \ref{l.boundaryleavesareperiodic} prove that $\rho$ satisfies the Axiom (A2) of Definition \ref{d.anosovlike}: 
\begin{lemm}\label{l.boundaryleavesarefiniteanddense}
      Up to the action of $\rho$, there are finitely many boundary leaves for $\mathcal{R}$ inside $\mathcal{F}^s$. Moreover, the union of all boundary leaves for $\mathcal{R}$ inside $\mathcal{F}^s$ forms a dense set in $\mathcal{P}$. 
\end{lemm}
Naturally, the previous lemma is also true for boundary leaves inside $\mathcal{F}^u$. 
\begin{proof}

The fact that the set of $\mathcal{F}^s$-boundary leaves for  $\mathcal{R}$ is finite up to the action of $\rho$ is an immediate consequence of the finiteness axiom (see Definition \ref{d.markovfamily}). Let us now prove that the union of all the $\mathcal{F}^s$-boundary leaves for $\mathcal{R}$ forms a dense set in $\mathcal{P}$.

Take $x\in \mathcal{P}$. By a repeated application of the strong finite return time axiom, there exists $R_0,R_1,R_2,...$ a sequence of rectangles in $\mathcal{R}$ such that $R_{i+1}\cap R_i$ is a horizontal subrectangle of $R_i$ for every $i\in \mathbb{N}$ and also such that $x\in \underset{i\in \mathbb{N}}{\cap}R_i$. By the expansiveness axiom,  $\underset{i\in \mathbb{N}}{\cap}R_i$ is an $\mathcal{F}^s$-leaf of $R_0$ containing $x$. Therefore, as $i\rightarrow +\infty$, the $\mathcal{F}^s$-boundary of $R_i$ gets arbitrarily close to $x$, which yields the desired result. 

\end{proof}

Notice that, since no rectangle in $\mathcal{R}$ can contain in its interior singular points of $\mathcal{F}^{s,u}$, we have that any singularity of $\mathcal{F}^{s,u}$ lies necessarily on the boundary of some rectangle in $\mathcal{R}$. Therefore its $\mathcal{F}^s$ or $\mathcal{F}^u$ leaf is periodic. Even more,  

\begin{lemm}\label{l.singularareperiodic} 
    If $p$ is a singular point of $\mathcal{F}^{s,u}$, then $p$ is periodic.
    \end{lemm}
\begin{proof}
Take $R\in \mathcal{R}$ containing $p$ and assume without any loss of generality that $p\in \partial^s R$. Thanks to Lemma \ref{l.boundaryleavesareperiodic}, this implies that $\mathcal{F}^s(p)$ is periodic. Any element of $G$, whose action preserves $\mathcal{F}^s(p)$ must fix $p$, the unique singular point inside $\mathcal{F}^s(p)$. This proves that $p$ is periodic.

\end{proof}
The previous lemma shows that $\rho$ satisfies Axiom (A3) of Definition \ref{d.anosovlike}.

Next, thanks to Lemma \ref{l.boundaryleavesareperiodic}, any strong Markovian action in $(\mathcal{P},\mathcal{F}^s,\mathcal{F}^u)$ admits infinitely many periodic leaves in both $\mathcal{F}^s$ and $\mathcal{F}^u$. In addition to the previous fact, we will later establish that strong Markovian actions admit infinitely many periodic points in $\mathcal{P}$. However, despite the plethora of periodic points for $\rho$, the set of fixed points of a non-trivial element in $G$ must be discrete in $\mathcal{P}$. More precisely,

\begin{lemm}\label{l.preserverectangleimpliestorsion}
Let $g\in G$ and $R\in \mathcal{R}$. If $\rho(g)(R)=R$, then $g=e$.
\end{lemm}
\begin{lemm}\label{l.twofixedpointsinthesamerectangle}
    For any $g\in G-\{e\}$ and any $R\in \mathcal{R}$ the homeomorphism $\rho(g)$ can fix at most one point inside $R$.
\end{lemm}
\begin{proof}[Proof of Lemma \ref{l.preserverectangleimpliestorsion}]
Take $g\in G-\{e\}$, $R\in \mathcal{R}$ and assume that  $\rho(g)(R)=R$. Endow both the restrictions of $\mathcal{F}^s$ and of $\mathcal{F}^u$ on $R$ with arbitrary orientations. Since $\rho$ is orientation preserving and $G$ has no torsion, after possibly replacing $g$ by $g^2\in G-\{e\}$, we can assume without any loss of generality that $\rho(g)$ preserves both of our previous orientations. 

Using our previous choices of orientations together with the fact that any two distinct predecessors (resp. successors) of $R$ have disjoint interiors, we get that the predecessors (resp. successors) of $R$ can be ordered from left to right (resp. bottom to top). Thanks to the previous fact and also to Remark \ref{r.precedentsuivant}, since $\rho(g)$ preserves $R$, it also preserves every successor and every predecessor of $R$. Similarly, an induction argument shows that, for every $n\in \mathbb{N}^*$, the homeomorphism $\rho(g)$ preserves every successor and every predecessor of $R$ of generation $n$.

Let $x\in R$ and $Q$ be a quadrant of $x$ such that a germ of $x$ in $Q$ is contained in $R$ (see Figure \ref{f.prooffixedpoints}). By a repeated application of Lemma \ref{l.existenceofpredecessors}, there exists a bi-infinite sequence of rectangles in $\mathcal{R}$ $$...,R_{-n}, \dots, R_{-1}, R_0 = R, R_1, \dots R_n,...$$
such that the rectangle $R_i$ contains a germ of $x$ inside $Q$ and also such that $R_{i+1}$ is a predecessor of $R_i$ for every $i\in \mathbb{Z}$. By definition, we have that, for every $i\in \mathbb{N}$, the rectangle $R_i$ (resp. $R_{-i}$) is a predecessor (resp. successor) of $R$ of generation $i$; hence $\rho(g)(R_i)=R_i$. Moreover, by the expansiveness axiom, we have that $\bigcap_{n \in \mathbb{Z}} R_n = \{ x \}$. This implies that $\rho(g)(x) = x$ and thus $\rho(g)$ fixes every point in $R$. By the same argument, we also get that if $R'$ is a predecessor or successor of $R$ of some generation, then $\rho(g)$ also fixes every point in $R'$.

\begin{figure}
    \centering
    \includegraphics[scale=0.7]{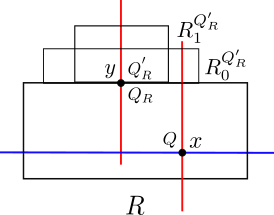}
    \caption{$\rho(g)$ fixes a neighborhood of $R$.}
    \label{f.prooffixedpoints}
\end{figure}
We would now like to prove that $\rho(g)$ fixes an open neighborhood of $R$. Indeed, let $S$ be an $\mathcal{F}^s$-boundary component of $R$, $y\in S$, $Q_R$ be a quadrant of $y$ such that $R$ contains a germ of $y$ in $Q_R$ and $Q_R'$ be a quadrant of $y$ intersecting $Q_R$ at its boundary and such that no germ of $y$ in $Q_R'$ is contained in $R$ (see Figure \ref{f.prooffixedpoints}). By a repeated application of Lemma \ref{l.existenceofpredecessors}, there exists a bi-infinite sequence of rectangles in $\mathcal{R}$
\[
...,R_{-n}^{Q_R'},....,R_{-1}^{Q_R'}, R_0^{Q_R'}, R_1^{Q_R'}, \dots, R_n^{Q_R'}, \dots
\]
such that $R_i^{Q_R'}$ contains a germ of $y$ in $Q_R'$ and also such that $R_{i+1}^{Q_R'}$ is a predecessor of $ R_{i}^{Q_R'}$ for every $i\in \mathbb{Z}$. By the expansivity axiom, for $n$ sufficiently big, either 
\begin{itemize}
   \item  $R_{n}^{Q_R'}$ intersects $R$ along a vertical subrectangle
   \item or $R_{n}^{Q_R'}$ and $R$ have disjoint interiors and an $\mathcal{F}^s$-boundary component $R_{n}^{Q_R'}$ is contained in $S$
\end{itemize}  In both of the above cases, let us show that $\rho(g)$ fixes every point in $R_{n}^{Q_R'}$. In the first case, thanks to Lemma \ref{l.npredecessor}, $R_{n}^{Q_R'}$ is a predecessor of some generation of $R$. Therefore, by our previous arguments, $\rho(g)$ fixes every point in $R_n^{Q_R'}$. In the second case, since $\rho(g)$ fixes $S$ we get that $\rho(g)(R_n^{Q_R'})$ shares an $\mathcal{F}^s$-boundary component with $R_n^{Q_R'}$. Even more, since the rectangles $\rho(g)(R_n^{Q_R'})\in \mathcal{R}$ and $\rho(g)(R)=R$ have disjoint interiors, we get that $\text{Int}(\rho(g)(R_n^{Q_R'}))\cap \text{Int}(R_n^{Q_R'})\neq \emptyset$. By the Markovian intersection axiom, this implies that $\rho(g)(R_n^{Q_R'})=R_n^{Q_R'}$ -we are utilizing here the fact that the intersection between two distinct rectangles of $\mathcal{R}$ was assumed to be \emph{non-trivial}, see Definition \ref{d.subrectangleplane}-. The arguments of this paragraph show that for $n$ sufficiently big $\rho(g)(R_n^{Q_R'})=R_n^{Q_R'}$. We deduce from the first part of our proof that $\rho(g)$ fixes every point of $R_i^{Q_R'}$ for every $i\in \mathbb{Z}$. 

By repeating the previous argument for the rest of the quadrants of $y$, we deduce that any rectangle in $\mathcal{R}$ intersecting $y$ is fixed by $\rho(g)$. Hence, any rectangle in $\mathcal{R}$ intersecting $R$ is fixed by $\rho(g)$. We conclude that there exists a neighborhood of $R$ that is fixed by $\rho(g)$. This implies that the set of fixed points of $\rho(g)$ is open and closed in $\mathcal{P}$, which proves that $\rho(g)=\text{id}_{
\mathcal{P}}$. This contradicts the faithfulness of $\rho$, as $g\neq e$, and finishes the proof of the lemma.
\end{proof}

\begin{proof}[Proof of Lemma \ref{l.twofixedpointsinthesamerectangle}]
    Suppose that there exist $g\in G-\{e\}$ and $R\in \mathcal{R}$ such that $\rho(g)$ fixes two distinct points $x,y\in R$. Since $G$ has no torsion, after possibly replacing $g$ by $g^n\in G-\{e\}$ for some $n\in\mathbb{N}$, we can assume that $\rho(g)$ preserves each $\mathcal{F}^s$-separatrix and each $\mathcal{F}^u$-separatrix of $x$ and $y$. As a result of the previous assumption, we obtain the following two facts: 
    
    \begin{itemize}
        \item $\rho(g)(\inte{R})\cap \inte{R}\neq \emptyset$. Hence, by Lemmas \ref{l.npredecessor} and \ref{l.preserverectangleimpliestorsion}, we have that $\rho(g)(R)$ is either a predecessor or a successor of $R$ of some non-zero generation

        \vspace{0.1cm}
        \item since $R$ is trivially bifoliated by $(\mathcal{F}^s,\mathcal{F}^u)$, we have that there exists $z\in \mathcal{F}^s(x)\cap \mathcal{F}^u(y)\cap R$ such that $\rho(g)(z)=z$
    \end{itemize} 
    
    We remark that the point $z$ could coincide with either $x$ or $y$ when $\mathcal{F}^u(x)=\mathcal{F}^u(y)$ or $\mathcal{F}^s(x)=\mathcal{F}^s(y)$ respectively. Assume without any loss of generality that $\mathcal{F}^u(x)\neq\mathcal{F}^u(y)$ (this is possible, thanks to Item (5) of Proposition \ref{p.propertiesoffoliformarkovianactions}) and thus that $z\neq x$. In that case, denote by $[x,z]^s$ the $\mathcal{F}^s$-segment in $R$ going from $x$ to $z$. After possibly replacing $g$ by its inverse $g^{-1}\in G-\{e\}$, we get that $\rho(g)$ fixes the points $x$ and $z$, preserves $[x,z]^s$ and that $\rho(g)(R)$ is a predecessor of $R$ of some non-zero generation. 
    
    Finally, by Remark \ref{r.precedentsuivant}, we have that for every $k\in\mathbb{N}$ the rectangle  $\rho(g^{k+1})(R)$ is a predecessor of some non-zero generation of $\rho(g^{k})(R)$ and thus also a predecessor of some generation of $R$. Using the fact that $\rho(g)([x,z]^s)=[x,z]^s$ together with the expansivity axiom, we have that $\underset{k\in \mathbb{N}}{\cap}\rho(g^{k})(R)$ is an $\mathcal{F}^u$-leaf in $R$ containing the $\mathcal{F}^s$-segment $[x,z]^s$, which is absurd. 
\end{proof}
\subsection{On the $\rho$-stabilizers of leaves in $(\mathcal{P},\mathcal{F}^s,\mathcal{F}^u)$}
Following our previous notations, our goal in this section is to show that the action $\rho$ satisfies Axiom (A1) of Definition \ref{d.anosovlike}. Our proof is divided into the following steps:

\begin{enumerate}[label=\textbf{Step \arabic*:}, leftmargin=*]
    \item Establish a necessary and sufficient condition for two points lying on the same periodic leaf of $\mathcal{F}^{s,u}$ to belong to a common rectangle of $\mathcal{R}$ (see Lemmas \ref{l.longrectanglesregularleaves} and \ref{l.longrectanglessingularleaves}).
    
    \item Combine this criterion with Lemma \ref{l.twofixedpointsinthesamerectangle} to show that every non-trivial element in the stabilizer of a leaf $l\in \mathcal{F}^{s,u}$ fixes exactly one periodic point in $l$ (see Lemmas \ref{l.longrectanglesregularleaves} and \ref{l.longrectanglessingularleaves}, as well as Corollary \ref{c.twofixedpointssameleaf}).
    
    \item Use this result together with a theorem of Solodov to prove that every periodic leaf of $\mathcal{F}^{s,u}$ contains a unique periodic point (see Lemma \ref{l.periodicleavescarryuniqueperiodicpoint}).
    
    \item Deduce that $\rho$ satisfies Axiom (A1) of Definition \ref{d.anosovlike} (see Lemma \ref{l.hyperbolicityforperiodicpoints}).
\end{enumerate}

In addition to the previous results, at the end of this section, we will prove a generalization of Lemmas \ref{l.longrectanglesregularleaves} and \ref{l.longrectanglessingularleaves} for non-periodic leaves in $\mathcal{F}^{s,u}$, a result that will be particularly useful to us in the final section of this paper. 

\begin{lemm}\label{l.longrectanglesregularleaves}
    Let $f$ be a regular periodic leaf in $\mathcal{F}^s$ and $D$ be the closure of a connected component of $\mathcal{P} - f$. For every $g\in \text{Stab}_{\rho}(f)$ we have that $\rho(g)$ fixes a point in $f$. Moreover, there exists a periodic point $p \in f$ such that:
    
   \begin{itemize}
       \item for every $x,y\in f$ belonging in the same $\mathcal{F}^s$-separatrix of $p$ in $f$ there exists $R\in \mathcal{R}$ containing the points $x,y$ and satisfying $\inte{R}\cap D\neq \emptyset$ 
       \item for any $x,y\in f$ belonging in different  $\mathcal{F}^s$-separatrices of $p$ in $f$ there exists $R\in \mathcal{R}$ containing the points $x,y$ and satisfying $\inte{R}\cap D\neq \emptyset$ if and only if there exists $R'\in \mathcal{R}$ containing $p$ and a neighborhood of $p$ inside $D$
   \end{itemize}
\end{lemm}
Naturally, the previous lemma applies for leaves in $\mathcal{F}^u$ too. 
\begin{proof}
  Consider $f$ and $D$ as in the above statement. Take $R_0\in \mathcal{R}$ intersecting $f$ and satisfying $\inte{R_0}\cap D\neq \emptyset$ (such a rectangle exists thanks to the strong finite return time axiom). Recall that by Item (5) of Proposition \ref{p.propertiesoffoliformarkovianactions}, $R_0\cap f$ is an $\mathcal{F}^s$-leaf of $R_0$. Consider now  $x\in f$ such that $R_0$ intersects non-trivially both $\mathcal{F}^s$-separatrices of $x$, which will by denoted by $\mathcal{F}^s_+(x)$ and $\mathcal{F}^s_-(x)$ respectively.  

By a repeated application of Lemma \ref{l.existenceofpredecessors}, there exists  $R_0,R_1,R_2,...$ a unique sequence of rectangles in $\mathcal{R}$ such that for every $i\in \mathbb{N}$ the rectangle $R_{i+1}$ is a successor of $R_i$,  contains $x$ and satisfies $\inte{R_i}\cap D\neq \emptyset$ (see Figure \ref{f.rectcontainstwopointscase1}). Thanks to Lemma \ref{l.npredecessor}, the previous sequence of rectangles consists of all the successors of some generation of $R_0$ that contain $x$ and whose interiors intersect $D$. 
    
    By construction, each $R_i$ intersects non-trivially both $\mathcal{F}^s_+(x)$ and $\mathcal{F}^s_-(x)$. Let $I_i^s:=R_i\cap f$. By Item (5) of Proposition \ref{p.propertiesoffoliformarkovianactions}, $I^s_i$ is an $\mathcal{F}^s$-leaf of $R_i$. Moreover, thanks to the Markovian intersection axiom, we have that $I^s_i\subseteq I^s_{i+1}$. Consider $J:=\underset{i\in\mathbb{N}}{\cup}I^s_i\cap f$. 

    We consider the following four cases. The first two are shown to be impossible, while the lemma is proved separately in each of the last two cases.
    
    \vspace{0.4cm}
    \textbf{Case 1: $J$ is not closed in $f$}  
    
    After possibly reversing the roles of $\mathcal{F}^s_+(x), \mathcal{F}^s_-(x)$, assume without any loss of generality that there exists a first element $t\in \mathcal{F}^s_+(x)$ that does not belong in $J$; hence $J\cap \mathcal{F}^s_+(x)=[x,t)^s$ (see Figure \ref{f.rectcontainstwopointscase1}). Consider $r\in \mathcal{R}$ containing $t$, intersecting non-trivially $J$ and satisfying $\inte{r}\cap D\neq \emptyset$ (such a rectangle exists thanks to the strong finite return time axiom). By the definition of $J$ and the expansivity axiom, when $i$ is sufficiently big, the rectangle $R_i$ is very thin in the $\mathcal{F}^u$ direction, intersects the interior of $r$ and does not contain $t$. Therefore, for $i$ sufficiently big the intersection between $R_i$ and $r$ can not be Markovian, which is absurd. We deduce that $J$ needs to be a closed subset of $f$.

\begin{figure}[h]
    \centering
    \includegraphics[scale=0.7]{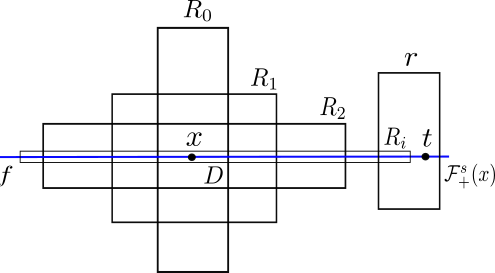}
    \caption{Case 1 contradicts the Markovian intersection axiom.}
    \label{f.rectcontainstwopointscase1}
\end{figure}
   
   \Needspace{4\baselineskip}
    \textbf{Case 2: $J$ is a compact interval in $f$}

   Consider $s,t \in f$ such that $J = [s,t]^s$.  In this case, there exists $i_0\in \mathbb{N}$ such that for every $i \geq i_0$, we have that $s,t \in \partial^u R_i$ and thus also that $R_i\cap f=J$. Since $R_{i+1}$ is a successor of $R_i$, the previous facts imply that  $\partial ^uR_{i+1}\subseteq \partial ^uR_i$. This contradicts the Markovian intersection axiom, according to which $R_{i+1}\cap R_i$ is a \emph{non-trivial} vertical subrectangle of $R_{i+1}$. We deduce that $J$ is a closed subset of $f$ that is not compact.

    \vspace{0.4cm}
    \textbf{Case 3: $J$ coincides with $f$}

   In this case, for any two points $y,y'\in f$ there exists $R\in \mathcal{R}$ containing the points $y,y'$ and satisfying $\inte{R}\cap D\neq \emptyset$. In particular, for any point $y\in f$ there exists $R\in \mathcal{R}$ containing $y$ and a neighborhood of $y$ inside $D$. Thanks to the previous facts, in order to prove the lemma in this case, it suffices to show that $\rho(g)$ fixes a point in $f$ for every $g\in \text{Stab}_{\rho}(f)$. Fix $g\in \text{Stab}_{\rho}(f)-\{e\}$. 
   
  First, notice that if $\rho(g)$ reverses the orientation of $f$, then clearly it fixes a point in $f$. On the contrary, if $\rho(g)$ preserves the orientation of $f$, then using the fact that $\rho$ acts by orientation preserving homeomorphisms on $\mathcal{P}$, we have that $$\rho(g)(D)=D$$ 
    
  Consider now $y\in f$ and $j\in \mathbb{N}$ sufficiently big so that $R_j$ contains both $y$ and $\rho(g)(y)\in f$ (such a $j\in \mathbb{N}$ exists, since $J=f$). By taking $j$ sufficiently large, we may further assume that $y, \rho(g)(y) \notin \partial^u R_j$; hence $R_j$ contains a neighborhood of $y$ in $D$ and a neighborhood of $\rho(g)(y)$ in $D$.

  By Lemma  \ref{l.preserverectangleimpliestorsion}, we have that $\rho(g)(R_j)\neq R_j$. Also, thanks to our previous hypotheses, we have that $\inte{R_j}\cap \rho(g)(\inte{R_j}) \neq \emptyset$. Therefore, the rectangles $R_j$ and $\rho(g)(R_j)$ intersect Markovianly. Assume, without loss of generality, that $\rho(g)(R_j)$ is a predecessor of $R_j$ of some generation. Since $g\in \mathrm{Stab}_{\rho}(f)$, we have $$ \rho(g)(R_j\cap f)=\rho(g)(R_j)\cap f \subset R_j\cap f$$ It follows that $\rho(g)$ fixes a point in $\rho(g)(R_j)\cap f$, which  concludes the proof the lemma in this case.

 \vspace{0.4cm}
    \textbf{Case 4: $J\neq f$, $J$ is closed and unbounded in $f$}

   In this case, by our construction of $J$, we have that $J$ is an unbounded and closed interval in $f$ with a unique extremity, which will be denoted by $p$ (see Figure \ref{f.rectcontainstwopointscase2}). Also, similarly to Case 2, there exists $i_0\in \mathbb{N}$ such that for every $i\geq i_0$ we have that $p\in \partial^u R_i$. 
   
   Let us now show there does not exist a rectangle $R\in \mathcal{R}$ containing $p$ and a neighborhood of $p$ inside $D$. Assume by contradiction that such a rectangle $R\in \mathcal{R}$ exists. Since $\inte{R_i}\cap D\neq \emptyset$ for every $i\in \mathbb{N}$, we have that $\inte{R_i}\cap \inte{R}\neq \emptyset$ for any $i\geq i_0$. Therefore, by Lemma \ref{l.npredecessor}, we get that $R$ is either a predecessor of some generation or a successor of some generation of $R_{i_0}$. If $R$ is a predecessor of some generation of $R_{i_0}$, then we have that $p\in \partial^u R$, which contradicts the fact that $r$ contains a neighborhood of $p$ inside $D$. If, on the other hand, $r$ is a successor of some generation of $R_{i_0}$, then $R$ contains $x$ and satisfies $\inte{r}\cap D\neq \emptyset$; hence, by our construction of the sequence $(R_n)_{n\in\mathbb{N}}$, we have that there exists $j>i_0$ such that $R=R_j$, which contradicts the fact that $p$ is an extremity of $J$.

   We deduce that there does not exist a rectangle $R\in \mathcal{R}$ containing $p$ and a neighborhood of $p$ inside $D$. In particular, if we denote by $\mathcal{F}^s_+(p)$ and $\mathcal{F}^s_-(p)$ the two $\mathcal{F}^s$-separatrices of $p$, then for any $y\in \mathcal{F}^s_+(p)-\{p\}$ and any  $y'\in \mathcal{F}^s_-(p)-\{p\}$ we have that there does not exist a rectangle $R\in \mathcal{R}$ containing $y,y'$ and satisfying $\inte{R}\cap D\neq \emptyset$. 

\begin{figure}[h!]
    \centering
    \includegraphics[scale=0.8]{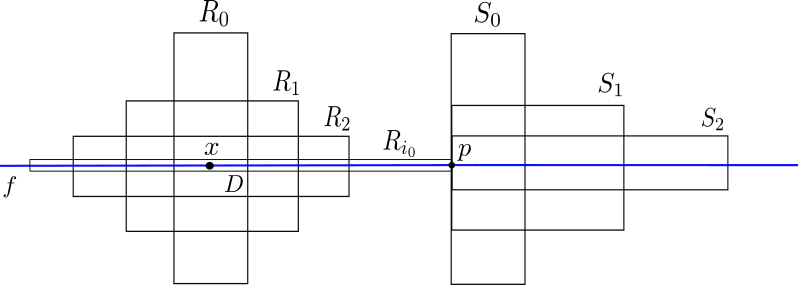}
    \caption{Case 2 contradicts Lemma \ref{l.twofixedpointsinthesamerectangle}.}
    \label{f.rectcontainstwopointscase2}
\end{figure}
   
Assume without loss of generality that $J=\mathcal{F}_-^s(p)$. Summarizing our previous arguments, we have established the following:

\begin{itemize}
\item For any two points $y,y'\in \mathcal{F}_-^s(p)$, there exists a rectangle $r\in \mathcal{R}$ such that $y,y'\in r$ and $\inte{r}\cap D\neq \emptyset$.

\item There does not exist a rectangle $r'\in \mathcal{R}$ containing $p$ and a neighborhood of $p$ inside $D$.

\item Any two points $y,y'\in f$ belonging to different $\mathcal{F}^s$-separatrices of $p$ cannot belong to a common rectangle $r''\in \mathcal{R}$ satisfying $\inte{r''}\cap D\neq \emptyset$.

\end{itemize}

To complete the proof of the lemma in this final case, it remains to establish the analogue of the first statement above for the separatrix $\mathcal{F}_+^s(p)$, to prove that $p$ is periodic, and to show that $\rho(g)$ fixes a point in $f$ for every $g\in \text{Stab}_{\rho}(p)$.

We begin with the first of these assertions. Take $S_0\in \mathcal{R}$ such that $S_0$ contains $p$, intersects non-trivially $\mathcal{F}^s_+(p)$ and satisfies $\inte{S_0}\cap D\neq \emptyset$. By a repeated application of Lemma \ref{l.existenceofpredecessors}, there exists  $S_0,S_1,S_2,...$ a unique sequence of rectangles in $\mathcal{R}$ such that for every $i\in \mathbb{N}$ the rectangle $S_{i+1}$ is a successor of $S_i$,  contains $p$ and satisfies $\inte{S_i}\cap D\neq \emptyset$. By the exact same argument as the one used for the sequence $(R_n)_{n\in \mathbb{N}}$, we have that $$ \underset{n\in \mathbb{N}}{\cup}S_n\cap f=\mathcal{F}^s_+(p)$$

   \noindent We deduce that for any $y,y'\in \mathcal{F}^s_+(p)$ there exists  $R\in \mathcal{R}$ containing $y,y'$ and satisfying $\inte{R}\cap D\neq \emptyset$.

  Next, let us show that $\rho(g)$ has a fixed point in $f$ for every $g\in \text{Stab}_{\rho}(f)$. Indeed, fix $g\in \text{Stab}_{\rho}(p)-\{e\}$. If $\rho(g)$ reverses the orientation of $f$, then clearly $\rho(g)$ has a fixed point in $f$. Suppose from now on that $\rho(g)$ preserves the orientation of $f$. Thanks to this fact and to the fact that $\rho$ acts by orientation preserving homeomorphisms on $\mathcal{P}$, we have that $$\rho(g)(D)=D$$ \noindent Even more, by our previous arguments, any two points $y,y'$ in $f$ are contained in a rectangle in $\mathcal{R}$ that intersects $D$ non-trivially if and only if $y,y'$ belong in the same $\mathcal{F}^s$-separatrix of $p$. As the Markovian family $\mathcal{R}$ and the half-plane $D$ are both preserved by $\rho(g)$, we get that $\rho(g)$ fixes $p$. 

  Finally, take any $g\in \text{Stab}_{\rho}(p)-\{e\}$. The arguments of the previous paragraph imply that $\rho(g^2)(p)=p$. Since $G$ has no torsion, we deduce that $p$ is periodic, which finishes the proof of the lemma. 
\end{proof}

\begin{lemm}\label{l.longrectanglessingularleaves}
    Let $f$ be a singular leaf in $\mathcal{F}^s$.  

   \begin{itemize}
        \item for every $x,y\in f$, if there exists $R\in \mathcal{R}$ containing $x,y$, then there exists $D$ the closure of a connected component of $\mathcal{P}-f$ such that $x,y\in D\cap f$ and $\inte{R}\cap D\neq \emptyset$.
    \end{itemize}
Consider $D$ to be the closure of a connected component of $\mathcal{P}-f$. Let $f_D:=f\cap D$. There exists a periodic point $p \in f_D$ such that:
    
   \begin{itemize}
       \item for every $x,y\in f_D$ belonging in the same $\mathcal{F}^s$-separatrix of $p$ in $f$ there exists $R\in \mathcal{R}$ containing the points $x,y$ and satisfying $\inte{R}\cap D\neq \emptyset$ 
       \item for any $x,y\in f_D$ belonging in different  $\mathcal{F}^s$-separatrices of $p$ in $f$ there exists $R\in \mathcal{R}$ containing the points $x,y$ and satisfying $\inte{R}\cap D\neq \emptyset$ if and only if there exists $R'\in \mathcal{R}$ containing $p$ and a neighborhood of $p$ in $D$
   \end{itemize}
\end{lemm}
The second and third items of the above lemma follow from a straightforward adaptation
of the proof of Lemma \ref{l.longrectanglesregularleaves}, while the first item is a direct consequence of Item (5)
of Proposition \ref{p.propertiesoffoliformarkovianactions} together with the fact that every rectangle is, by definition, trivially bifoliated by $(\mathcal{F}^s,\mathcal{F}^u)$. We will therefore omit the proof of the previous lemma.

\begin{comment}
By slightly modifying our proof of Lemma \ref{l.longrectanglesregularleaves}, we will now generalize Lemma \ref{l.longrectanglesregularleaves} for the leaves in $\mathcal{F}^s$ that do not contain periodic points. 

\end{comment}

\begin{coro}\label{c.twofixedpointssameleaf}
     For any $g\in G-\{e\}$ and any $f\in \mathcal{F}^s$ the homeomorphism $\rho(g)$ can fix at most one point inside $f$.
\end{coro}
Naturally, the previous lemma applies for leaves in $\mathcal{F}^u$ too. 
\begin{proof}
   Suppose by contradiction that there exist  $f\in \mathcal{F}^s$ and $g\in \text{Stab}_{\rho}(f)-\{e\}$ such that $\rho(g)$ fixes two distinct points in $f$, say $x$ and $y$. 
    
    \vspace{0.5cm}
    \textbf{Case 1: $f$ is a regular leaf of $\mathcal{F}^s$} 

    Consider $D$ the closure of a connected component of $\mathcal{P}-f$. Using the fact that $G$ has no torsion, by possibly replacing $g$ by $g^2$, assume without any loss of generality that $\rho(g)$ preserves the orientation of $f$. Thanks to this fact that to the fact that $\rho$ acts by orientation preserving homeomorphisms on $\mathcal{P}$, we have that $$\rho(g)(D)=D$$ Next, recall that by Lemma \ref{l.longrectanglesregularleaves}, there exists a periodic point $p$ in $f$ such that any two points in $f$ are contained in a rectangle $R\in \mathcal{R}$ satisfying $\inte{R}\cap D\neq \emptyset $ if and only if they lie in the same $\mathcal{F}^s$-separatrix of $p$. By the above characterization, 
    
    \begin{itemize}
        \item  since the Markovian family $\mathcal{R}$ and the half-plane $D$ are both preserved by $\rho(g)$, we have that $\rho(g)(p)=p$.

        \vspace{0.1cm}
        \item there exist $R_x,R_y\in \mathcal{R}$ such that $p,x\in R_x$ and $p,y\in R_y$.
    \end{itemize}
    
    This forces $\rho(g)$ to admit two fixed points in either $R_x$ or $R_y$, which contradicts Lemma \ref{l.twofixedpointsinthesamerectangle} and leads to an absurd.

    \vspace{0.5cm}
    \textbf{Case 2: $f$ is a singular leaf of $\mathcal{F}^s$}

    In this case, $\rho(g)$ must fix the unique singular point in $f$. We can therefore assume without any loss of generality that $y$ coincides with the singular point in $f$. By adapting the argument used in Case 1, one may verify that this case also leads to a contradiction.
\end{proof}
\begin{coro}\label{c.uniqueperiodicpointinsingular}
    The only periodic point contained in a singular leaf $f$ of $\mathcal{F}^{s,u}$ is the unique singular point in $f$.  
\end{coro}
\begin{proof}
    Take $f$ a singular leaf of $\mathcal{F}^{s,u}$. Assume that there exists $p\in f$ a non-singular periodic point. If this were true, there would exist $g\in G-\{e\}$ such that $\rho(g)(p)=p$. This is impossible, as $\rho(g)$ would also fix the unique singular point in $f$, which would contradict Corollary \ref{c.twofixedpointssameleaf}. 
\end{proof}
The previous corollary can be generalized for all periodic leaves in $\mathcal{F}^{s,u}$: 

\begin{lemm}\label{l.periodicleavescarryuniqueperiodicpoint}
    Let $f$ be a periodic leaf in $\mathcal{F}^s$ or $\mathcal{F}^u$. There exists a unique periodic point of $\rho$ inside $f$. 
\end{lemm}
The following argument is taken from \cite{circleatinfinity} (see Lemma 2.6), where the above lemma is proved in the setting of Anosov-like actions. Since the same proof applies to Markovian actions, we reproduce it here for completeness.
\begin{proof}[Proof of Lemma \ref{l.periodicleavescarryuniqueperiodicpoint}]
If $f$ is singular, then the previous result follows from Corollary  \ref{c.uniqueperiodicpointinsingular}. We can therefore assume that $f$ is a regular leaf of $\mathcal{F}^s$ or $\mathcal{F}^u$. Parameterize $f$ and consider the action of $\text{Stab}_{\rho}(f)$ on $f$ as an action of $\text{Stab}_{\rho}(f)$ on $\mathbb{R}$. Thanks to Lemma \ref{l.longrectanglesregularleaves} and to Corollary \ref{c.twofixedpointssameleaf}, each non-trivial element of $\mathrm{Stab}_{\rho}(f)$ has a unique fixed point in $f$. By Solodov's theorem (see Theorem 2.2.36 of \cite{Navas}), if the action of $\text{Stab}_{\rho}(f)$ on $\mathbb{R}$ has no global fixed point, then it is semi-conjugate to a subgroup of affine transformations of $\mathbb{R}$. But such a group always contains a translation, which contradicts the existence of a fixed point for every element of $\text{Stab}_{\rho}(f)$. We deduce that $\text{Stab}_{\rho}(f)$ has a global fixed point in $f$, which proves the desired result. 
\end{proof}

We are now ready to prove that $\rho$ satisfies Axiom (A1) of Definition \ref{d.markovianaction}. 

\begin{lemm} \label{l.hyperbolicityforperiodicpoints}
Let $p \in \mathcal{P}$ be a periodic point of $\rho$ and $g\in \text{Stab}_{\rho}(p)-\{e\}$. After possibly replacing $g$ with $g^{-1}$, we have that 
$$\forall x \in \mathcal{F}^s(p), \quad \rho(g^n)(x) \underset{n\rightarrow +\infty}{\longrightarrow} p $$
$$\forall y \in \mathcal{F}^u(p), \quad \rho(g^{-n})(y) \underset{n\rightarrow +\infty}{\longrightarrow} p$$
\end{lemm}

\begin{proof}
Fix $f^s$ and $f^u$ an $\mathcal{F}^s$-separatrix and an $\mathcal{F}^u$-separatrix of $p$ bounding a quadrant $Q$ of $p$ inside $\mathcal{P}$. Consider $x\in f^s$, $y\in f^u$, and $R\in \mathcal{R}$ containing $p$ and a germ of $p$ inside $Q$ (such a rectangle exists, thanks to the strong finite return axiom). Notice that, since $G$ has no torsion, it suffices to prove the lemma for some power of $g$. By the previous fact, after replacing $g$ by a suitable power $g^k$, where $k\in\mathbb{N}$, we may assume that $\rho(g)$ preserves every $\mathcal{F}^s$- and $\mathcal{F}^u$-separatrix of $p$. Moreover, after possibly replacing $g$ by $g^{-1}$, we may further assume that $\rho(g)(R)$ is a predecessor of some generation of $R$ (see Remark \ref{r.precedentsuivant}).

\vspace{0.5cm}
    \textbf{Claim. There exists $N\in\mathbb{N}$ such that $y\in \rho(g^N)(R)$ and $x\in \rho(g^{-N})(R)$}
\begin{proof}[Proof of the Claim]
  
  Thanks to Lemma \ref{l.periodicleavescarryuniqueperiodicpoint}, the point $p$ is the unique periodic point of $\rho$ in both $\mathcal{F}^s(p)$ and $\mathcal{F}^u(p)$. Using this fact, together with Lemmas \ref{l.longrectanglesregularleaves} and \ref{l.longrectanglessingularleaves}, we have that there exist $R_x,R_y\in\mathcal{R}$ such that $p,x\in R_x$, $p,y\in R_y$ and also such that $R_x$ and $R_y$ contain a germ of $p$ inside $Q$ (see Figure \ref{f.contractexpand}).

\begin{figure}[h!]
    \centering
    \includegraphics[scale=0.55]{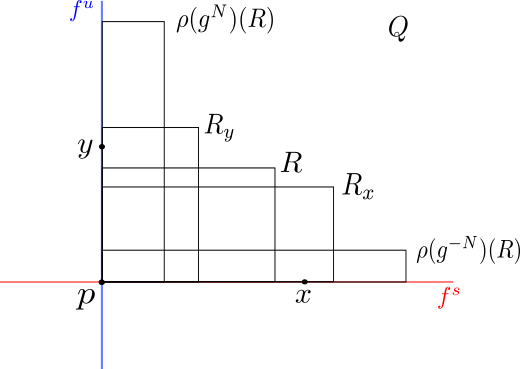}
    \caption{Lemmas~\ref{l.longrectanglesregularleaves} and \ref{l.longrectanglessingularleaves} imply that $\rho(g)$ exhibits hyperbolic behavior near $p$.}
    \label{f.contractexpand}
\end{figure}
  Since $R,R_x,R_y$ all contain a germ of $p$ inside $Q$, $$\inte{R_x}\cap \inte{R}\neq \emptyset \text{ and }\inte{R_y}\cap \inte{R}\neq \emptyset$$ Even more, since  $\rho(g)(p)=p$ and $\rho(g)$ preserves the $\mathcal{F}^s$- and $\mathcal{F}^u$-separatrices of $p$, we have that $\rho(g)(Q)=Q$ and also that 
   \begin{equation}\label{eq.alwaysinter}
      \inte{R_x}\cap \rho(g^n)(\inte{R})\neq \emptyset \text{ and } \inte{R_y}\cap \rho(g^n)(\inte{R})\neq \emptyset
  \end{equation} for every $n\in \mathbb{Z}$. Notice now that for any $n\in \mathbb{N}$, thanks to Remark \ref{r.precedentsuivant}, the rectangle $\rho(g^n)(R)$ (resp. $\rho(g^{-n})(R)$) is a predecessor (resp. successor) of $R$ whose generation increases with $n$. By (\ref{eq.alwaysinter}) and Remark \ref{r.precedentsuivant}, there exists $N\in \mathbb{N}$ such that $\rho(g^N)(R)$ is a predecessor of some generation of $R_y$ and $\rho(g^{-N})(R)$ is a successor of some generation of $R_x$ (see Figure \ref{f.contractexpand}). Thanks to Item (5) of Proposition \ref{p.propertiesoffoliformarkovianactions}, for any such $N$ we have that $y\in \rho(g^N)(R)$ and that $x\in \rho(g^{-N})(R)$, which proves the desired result. 
\end{proof}

Let us now resume the proof of the lemma. Recall that by the expansivity axiom $\underset{n\leq N}{\cap} \rho(g^n)(R)=\mathcal{F}^s(p)\cap \rho(g^N)(R)$ and $\underset{n\geq -N}{\cap} \rho(g^n)(R)=\mathcal{F}^u(p)\cap \rho(g^{-N})(R)$. The previous fact, together with the invariance of $f^s$ and $f^u$ under $\rho(g)$ and the equalities
$\mathcal{F}^s(p)\cap \mathcal{F}^u(y)=\mathcal{F}^u(p)\cap \mathcal{F}^s(x)=\{p\}$,
implies that
  $$\rho(g^{-n})(y) \underset{n\rightarrow +\infty}{\longrightarrow} p \text{  and  } \rho(g^n)(x) \underset{n\rightarrow +\infty}{\longrightarrow} p $$

\noindent which proves the lemma for every point in $f^s\cup f^u$. 

Finally, let $\widetilde{f^s}$ be an $\mathcal{F}^s$-separatrix of $p$ such that $\widetilde{f^s}$ and $f^u$ bound a quadrant of $p$, say $\widetilde{Q}$. Notice that thanks to the Markovian intersection axiom and since $\rho(g)$ acts as a topological expansion on $f^u$, we have that for every $\widetilde{R}\in \mathcal{R}$ containing $p$ and a germ of $p$ inside $\widetilde{Q}$, the rectangle $\rho(g)(\widetilde{R})$ is a predecessor of some generation of $\widetilde{R}$. Using our previous arguments, one may verify that $\rho(g)$ acts on $\widetilde{f^s}$ as a topological contraction. We obtain the desired result by a finite iteration of this process. 

\end{proof}
We conclude this section by proving the following generalization of  Lemmas~\ref{l.longrectanglesregularleaves} and~\ref{l.longrectanglessingularleaves} for non-periodic regular leaves of $\mathcal{F}^{s,u}$. 

\begin{lemm}\label{l.longrectanglesregularleavesnoperiodicpoints}
    Let $f$ be a regular leaf in $\mathcal{F}^s$ or $\mathcal{F}^u$ that does not contain a periodic point and $D$ be the closure of a connected component of $\mathcal{P}-f$. We have that for any $x,y\in f$ there exists $R\in \mathcal{R}$ containing  $x,y$ and such that $\inte{R}\cap D\neq \emptyset$. 
\end{lemm}
The proof of this lemma is an adaptation of our proof of Lemma \ref{l.longrectanglesregularleaves}.
\begin{proof}
   Fix $x,y \in f$ and denote by $f_+$ the $\mathcal{F}^s$-separatrix of $x$ containing $y$. Take $R_0\in\mathcal{R}$ such that $R_0$ contains $x$, intersects $f_+$ non-trivially and satisfies $\inte{R_0}\cap D\neq \emptyset$ (such a rectangle exists, thanks to the strong finite return time axiom). By a repeated application of strong finite return axiom, there exists  $R_0,R_1,R_2...$ a sequence of rectangles in $\mathcal{R}$ such that for every $i\in \mathbb{N}$ the rectangle $R_{i+1}$ is a successor of some generation of $R_i$, contains $x$ and satisfies $\inte{R_i}\cap D\neq \emptyset$.  
    
    By construction, each $R_i$ intersects $f_+$ non-trivially.  Let $I_i^s:=R_i\cap f$. By Item (5) of Proposition \ref{p.propertiesoffoliformarkovianactions}, $I^s_i$ is an $\mathcal{F}^s$-leaf of $R_i$. Moreover, thanks to the Markovian intersection axiom, we have that $I^s_i\subseteq I^s_{i+1}$. Consider $J:=\underset{i\in\mathbb{N}}{\cup}I^s_i\cap f_+$. 
    
    \vspace{0.5cm}
    \textbf{Case 1: $J\neq f_+$ and $J$ is not closed in $f_+$}  
    
    Assume that there exists a first element $t\in f_+$ that does not belong in $J$; hence $J=[x,t)^s$. Consider $r\in \mathcal{R}$ containing $t$, intersecting non-trivially $J$ and satisfying $\inte{r}\cap D\neq \emptyset$ (such a rectangle exists thanks to the strong finite return time axiom). By the definition of $J$ and the expansivity axiom, when $i$ is sufficiently big, the rectangle $R_i$ is very thin in the $\mathcal{F}^u$ direction,  intersects the interior of $r$ and does not contains $t$ (see Figure \ref{f.rectcontainstwopointscase1}). We deduce that for $i$ sufficiently big the intersection between $R_i$ and $r$ can not be Markovian, which is absurd. 

    \vspace{0.5cm}
    \textbf{Case 2: $J\neq f_+$ and $J$ is a closed subset of $f_+$}

    Denote by $p\in f_+$ the extremity of $J$ which is not $x$. In this case, we have that there exists $i_0\in\mathbb{N}$ such that $p\in \partial^u R_i$ for every $i\geq i_0$. Thanks to the finiteness axiom, after possibly passing to a subsequence, assume  that all the $R_i$ are contained in the same $\rho$-orbit of rectangles. Take $g_i\in G$ such that $\rho(g_i)(R_i)=R_{i+1}$. Passing to a further subsequence if necessary, we can assume that $\rho(g_i)(R_{i}\cap \mathcal{F}^u(p))=R_{i+1}\cap \mathcal{F}^u(p)$ for every $i\geq i_0$. 
    
    Thanks to the Markovian intersection axiom, since $R_{i+1}$ is a successor of some generation of $R_i$, $$\rho(g_i)(R_{i}\cap \mathcal{F}^u(p))=R_{i+1}\cap \mathcal{F}^u(p)\subsetneq R_{i}\cap \mathcal{F}^u(p)$$  for every $i\geq i_0$; hence, $\rho(g_i)$ has a fixed point inside $R_{i+1}\cap \mathcal{F}^u(p)\subset \mathcal{F}^u(p)$  for every $i\geq i_0$. Recall now that by Lemma \ref{l.periodicleavescarryuniqueperiodicpoint}, there is a unique periodic point inside $\mathcal{F}^u(p)$. Therefore, the $\rho(g_i)$ with $i\geq i_0$ have a common fixed point in $\mathcal{F}^u(p)$, which lies inside $\underset{i\geq i_0}{\cap}R_{i}\cap \mathcal{F}^u(p)$. By the expansivity axiom, the previous set is equal to $\{p\}$ and thus $p\in f$ is a periodic point. This is impossible, since $f$ by hypothesis does not contain periodic points. 
    
    We conclude that $J=f_+$, which proves the existence of a rectangle $R\in \mathcal{R}$ containing $x,y$ and such that $\inte{R}\cap D\neq \emptyset$.    
\end{proof}

\subsection{On the non-separated leaves in $\mathcal{F}^s$ and $\mathcal{F}^u$}

In this section, we prove that $\rho$ satisfies Axiom (A4) of Definition \ref{d.anosovlike}. This result will be deduced from the following lemma.

\medskip

\noindent\textbf{Notation convention.} Consider $x\in \mathcal{P}$ and $y\in \mathcal{F}^{s}$. We denote by $[x,y]^{s}$ the $\mathcal{F}^{s}$-segment in $\mathcal{F}^{s}(x)$ between $x$ and $y$. Moreover, we denote by $(x,y]^{s}$ (resp. $[x,y)^{s}$) the same segment with the endpoint $x$ (resp. $y$) removed. We employ similar notation for $\mathcal{F}^u$-segments in $\mathcal{P}$. By abuse of notation, we will sometimes regard these segments as oriented from $x$ to $y$.

\begin{lemm}\label{l.nonseparationlemma}
   Take $l$ and $l'$ two non-separated leaves in $\mathcal{F}^s$. There exist $x\in l$, $x'\in l'$, $[x,x_0]^u$ an $ \mathcal{F}^u$-segment inside $\mathcal{F}^u(x)$,  $[x',x'_0]^u$ an $ \mathcal{F}^u$-segment inside $\mathcal{F}^u(x')$ and a homeomorphism $h: (x,x_0]^u\rightarrow (x',x'_0]^u$ such that 
   
   \begin{enumerate}
       \item for any $y\in (x,x_0]^u$ we have that $$\mathcal{F}^s(y)\cap (x',x'_0]^u= \{h(y)\}  $$ 
       \item $l\cap [x',x'_0]^u=\emptyset$ and $l'\cap [x,x_0]^u=\emptyset$
       \item the set $\underset{y\in (x,x_0]^u}\bigcup [y,h(y)]^s$ does not contain any singular points of $\mathcal{F}^s$ or $\mathcal{F}^u$ and is trivially bifoliated by $\mathcal{F}^s$ and $\mathcal{F}^u$
       \item if $U$ (resp. $U'$) is the connected component of $\mathcal{P}-l$ (resp. $\mathcal{P}-l'$) containing $l'$ (resp. $l$), then for every  $y\in U\cap U'$ we have that  $$\mathcal{F}^u(y)\cap l\neq \emptyset \implies \mathcal{F}^u(y)\cap l'= \emptyset$$
   \end{enumerate}
\end{lemm}

\begin{proof}
    Consider $l$ and $l'$ two non-separated leaves in $\mathcal{F}^s$. Fix $x\in l$ and $x'\in l$ such that both $\mathcal{F}^u(x)$ and $\mathcal{F}^u(x')$ are regular leaves in $\mathcal{F}^u$. If $l$ (resp. $l'$) is singular, we choose $x$ (resp. $x'$) to be the unique singular point in $l$ (resp. $l'$). Since $l$ and $l'$ are non-separated, there exist two connected components
$U \subset \mathcal{P} - l \text{ and } U' \subset \mathcal{P} - l'$ and two sequences of points in $\mathcal{P}$, say  $(x_n)_{n\in\mathbb{N}}$ and $(x_n')_{n\in\mathbb{N}}$,  
such that: 
    \begin{itemize}
        \item $x_0\in \mathcal{F}^u(x)\cap U$ and $x'_0\in \mathcal{F}^u(x')\cap U'$

        \vspace{0.1cm}
        \item $x_k\in [x,x_0]^u$, $x_k'\in [x',x'_0]^u$ for every $k\in\mathbb{N}$
        \vspace{0.1cm}
         \item $\mathcal{F}^s(x_k)$ is a regular leaf of $\mathcal{F}^s$ such that $\mathcal{F}^s(x_k)\cap [x',x'_0]^u=\{x_k'\}$ for every $k\in\mathbb{N}$ 
         \vspace{0.1cm}
        \item $x_n\underset{n\rightarrow +\infty}{\longrightarrow}x$, $x'_n\underset{n\rightarrow +\infty}{\longrightarrow}x'$
        \vspace{0.1cm}
       \item $(x_n)_{n\in\mathbb{N}}$ and $(x_n')_{n\in\mathbb{N}}$ are monotonous. In other words, $x_{k+1}\in [x,x_k]^u$, $x'_{k+1}\in [x',x'_k]^u$ for every $k\in\mathbb{N}$
    \end{itemize}
      We remark that in the last of the above properties  $\mathcal{F}^s(x_k)\cap [x',x'_0]^u$ consists indeed of a single point, thanks to Proposition \ref{p.propertiesoffoliformarkovianactions}.
      
     For every $k\in \mathbb{N}$ consider $\gamma_k$ the closed curve obtained by the juxtaposition of $[x_k,x_k']^s$ of followed by $[x_k',x_{k+1}']^u$, followed by $[x_{k+1}',x_{k+1}]^s$ and finally followed by $[x_{k+1},x_k]^u$. By Item (4) of Proposition \ref{p.propertiesoffoliformarkovianactions}, the curve $\gamma_k$ is simple; hence, it  bounds a disk in $\mathcal{P}$. Even more, since $\gamma_k$ consists of two $\mathcal{F}^s$-segments and two $\mathcal{F}^u$-segments, by Item (2) of Proposition \ref{p.propertiesoffoliformarkovianactions}, we get that every leaf in $\mathcal{F}^s$ (resp. $\mathcal{F}^u$) intersecting one $\mathcal{F}^u$ (resp. $\mathcal{F}^s$) segment of $\gamma_k$ must also intersect the other. It follows that the disk bounded by $\gamma_k$ is trivially bifoliated by $\mathcal{F}^s$ and $\mathcal{F}^u$ and thus is a rectangle in $\mathcal{P}$. 
     
     Using the fact that all the $\gamma_k$ bound rectangles, one easily deduces Item (1) of the above lemma. Item (3) follows by the same argument, together with the facts that $\mathcal{F}^u(x)$, $\mathcal{F}^u(x')$, and all the $\mathcal{F}^s(x_k)$ are regular leaves of $\mathcal{F}^{s,u}$, and that no rectangle in $\mathcal{P}$ can contain singularities of $\mathcal{F}^s$ or $\mathcal{F}^u$ in its interior. Concerning Item (2), assume that $l\cap [x',x'_0]^u\neq \emptyset$. Since $l\neq l'$, we have that $l$ intersects $(x',x'_0]^u$. By our previous arguments, this implies that $l$ also intersects $(x,x_0]^u$, which contradicts Item (4) of Proposition \ref{p.propertiesoffoliformarkovianactions}. 

     Finally, take $y\in U\cap U'$ and assume by contradiction that $\mathcal{F}^u(y)\cap l\neq \emptyset$ and that $\mathcal{F}^u(y)\cap l'\neq \emptyset$. The previous facts imply that for any $z\in (x,x_0]^u$ that is sufficiently close to $x$, we have that $\mathcal{F}^s(z)$ intersects twice $\mathcal{F}^u(y)$, which contradicts Item (4) of Proposition \ref{p.propertiesoffoliformarkovianactions}.
\end{proof}
\begin{lemm}\label{l.nonseparatedimpliesfixed}
     Take $l,l'\in \mathcal{F}^s$. If $l$ and $l'$ are not separated, then there exists $g\in G-\{e\}$ such that $\rho(g)(l)=l$.
\end{lemm}

The above lemma was originally proved in \cite{Fe} for the natural action of $\pi_1(M)$ on the bifoliated plane of an Anosov flow $(M^3,\Phi)$. Our proof below follows the general strategy of the argument in \cite{Fe}.
\begin{proof}[Proof of Lemma \ref{l.nonseparatedimpliesfixed}]
    Consider $l$ and $l'$ two non-separated leaves in $\mathcal{F}^s$. Thanks to Lemma \ref{l.nonseparationlemma}, there exist $x\in l$, $x'\in l'$, $[x,x_0]^u$ a segment inside $\mathcal{F}^u(x)$,  $[x',x'_0]^u$ a segment inside $\mathcal{F}^u(x')$ and a homeomorphism $h: (x,x_0]^u\rightarrow (x',x'_0]^u$ such that for any $z\in (x,x_0]^u$ $$\mathcal{F}^s(z)\cap (x',x'_0]^u= \{h(z)\}  $$ 

Denote by $\mathfrak{R}$ the set  $\underset{y\in (x,x_0]^u}\bigcup [y,h(y)]^s$ and by $\overline{\mathfrak{R}}$ the closure of $\mathfrak{R}$ in $\mathcal{P}$. By Lemma \ref{l.nonseparationlemma}, we may assume that $\mathfrak{R}$ contains no singularities of $\mathcal{F}^{s,u}$ and that it is trivially bifoliated by $\mathcal{F}^s$ and $\mathcal{F}^u$.

\vspace{0.5cm}
\textbf{Claim 1.} There exists a unique point $p_0\in [x_0,x_0']^s$ such that for any point $z\in [x_0,x_0']^s$ $$\mathcal{F}^u(z)\cap l\neq \emptyset \Longleftrightarrow z\in [x_0,p_0)^s$$
Moreover, for any $ z\in [x_0,p_0)^s$ there exists an $\mathcal{F}^u$-segment in $\overline{\mathfrak{R}}$ going from $z$ to some point in $l$. 
\vspace{0.2cm}
\begin{proof}[Proof of Claim 1]

Indeed, suppose that there exists $z\in [x_0,x_0']^s$, whose $\mathcal{F}^u$-leaf intersects $l$, but for which there does not exist an $\mathcal{F}^u$-segment in $\overline{\mathfrak{R}}$ going from $z$ to some point in $l$. In that case, thanks to the fact that $\mathfrak{R}$ is trivially bifoliated, there exists a leaf in $\mathcal{F}^s$ that intersects $\mathcal{F}^u(z)$ more than once, which contradicts Item (4) of Proposition \ref{p.propertiesoffoliformarkovianactions} (see Figure \ref{f.nonseparated1}). We deduce that for any point $z\in [x_0,x_0']^s$ we have that $\mathcal{F}^u(z)\cap l \neq \emptyset$ if and only if there exists an $\mathcal{F}^u$-segment in $\overline{\mathfrak{R}}$ going from $z$ to $l$. 

  \begin{figure}[h!]
    \centering
    \includegraphics[width=0.5\linewidth]{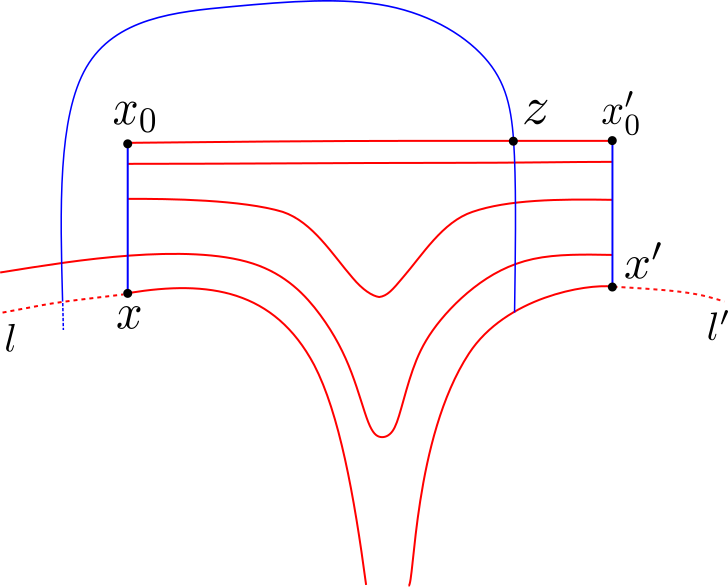}
    \caption{The $\mathcal{F}^u$-leaf of $z$ can not intersect both $l$ and $l'$.}
    \label{f.nonseparated1}
\end{figure}

Next, since $\mathfrak{R}$ contains no singularities of $\mathcal{F}^{s,u}$ and the foliations $\mathcal{F}^s$ and $\mathcal{F}^u$ are transverse, the set of points in $[x_0,x_0']^s$ whose $\mathcal{F}^u$-leaf intersects $l$ is an open subset of $[x_0,x_0']^s$. 
Starting from $x_0$, consider $p_0$ to be the first point in $[x_0,x_0']^s$ such that $\mathcal{F}^u(p_0)\cap l=\emptyset$ (such a point exists thanks to Item (4) of Lemma \ref{l.nonseparationlemma} and to the fact that $\mathcal{F}^u(x_0')\cap l'\neq\emptyset$). By Proposition \ref{p.propertiesoffoliformarkovianactions}, $\mathcal{F}^u(p_0)\cap [x_0,x_0']^s=\{p_0\}$; hence, there exist $V,V'$ two distinct connected components of $\mathcal{P}-\mathcal{F}^u(p_0)$ such that $[x_0,p_0]^s\subset \clos{V}$ and $[p_0,x_0']^s\subset \clos{V'}$. 

On the one hand, since $\mathcal{F}^u(x_0)\cap \mathcal{F}^u(p_0)=\emptyset$, $\mathcal{F}^u(x_0)\cap l\neq \emptyset$, and $\mathcal{F}^u(p_0)\cap l=\emptyset$, it follows that $l\subset V$. On the other hand, for any $w\in (p_0,x_0']^s$, the fact that $\mathcal{F}^u(p_0)\cap [x_0,x_0']^s=\{p_0\}$ implies that $\mathcal{F}^u(w)\cap \mathcal{F}^u(p_0)=\emptyset$; hence $\mathcal{F}^u(w)\subset V'$. We deduce that the $\mathcal{F}^u$-leaf of any point in $(p_0,x_0']^s$ cannot intersect $l$, which proves the desired result.
  
\end{proof}

%In addition to the existence of $p_0$, using the same arguments as in the proof of Claim 1, one can prove that here exists $q_0\in [x_0,x_0']^s$ such that $\mathcal{F}^u(q_0)\cap l' = \emptyset $ and also such that for every $y\in [x_0,x'_0]^s$ $$\mathcal{F}^u(y)\cap l'\neq \emptyset \Longleftrightarrow y\in (q_0,x'_0]^s$$ We remark that, thanks to Item (4) of Lemma \ref{l.nonseparationlemma}, $p_0\notin (q_0,x'_0]^s$ and $q_0\notin [x_0,p_0)^s$. Despite the previous fact, it is possible that the points $p_0$ and $q_0$ coincide.   

\vspace{0.5cm}
\textbf{Claim 2.} We have that $\mathcal{F}^u(p_0)\cap \mathfrak{R}$ is an $\mathcal{F}^u$-separatrix of $p_0$. 
%Similarly, $\mathcal{F}^u(q_0)\cap \mathfrak{R}$ is an $\mathcal{F}^u$-separatrix of $q_0$.
\vspace{0.2cm}

\begin{proof}[Proof of Claim 2.]
%We will prove that $\mathcal{F}^u(p_0)\cap \mathfrak{R}$ is an $\mathcal{F}^u$-separatrix of $p_0$. The fact that $\mathcal{F}^u(q_0)\cap \mathfrak{R}$ is an $\mathcal{F}^u$-separatrix of $q_0$ follows from a similar argument. 

Thanks to Item (4) of Proposition \ref{p.propertiesoffoliformarkovianactions} and to the fact that the interior of $\mathfrak{R}$ is trivially bifoliated we have that 

\begin{itemize}
    \item there exists a unique $\mathcal{F}^u$-separatrix of $p_0$ intersecting the interior of $\mathfrak{R}$ (see Figure \ref{f.nonseparated2}). Denote by $\mathcal{F}_-^u(p_0)$ the previous separatrix.
    \item $\mathcal{F}_-^u(p_0)\cap \mathfrak{R}$ is an open and connected subset of $\mathcal{F}_-^u(p_0)$ containing $p_0$
\end{itemize} 

  \begin{figure}[h!]
    \centering
    \includegraphics[width=0.5\linewidth]{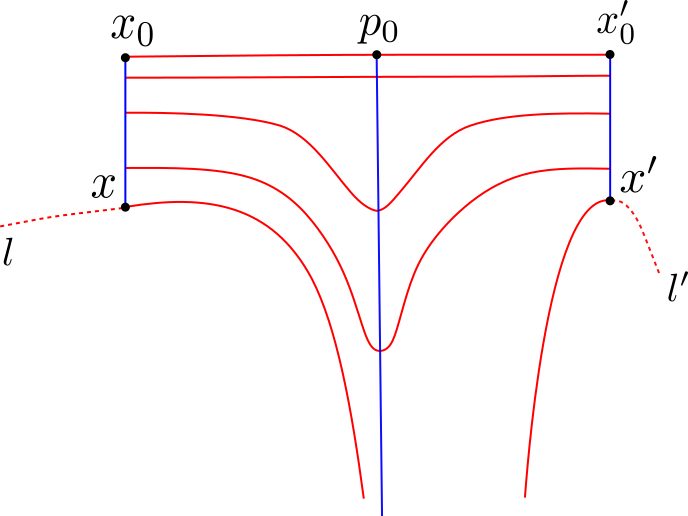}
    \caption{The above $\mathcal{F}^u-$separatrix of $p_0$ is contained in $\mathfrak{R}$.}
    \label{f.nonseparated2}
\end{figure}

Assume that $\mathcal{F}^u(p_0)\cap \mathfrak{R}\neq \mathcal{F}^u(p_0)$. By our previous arguments, this implies that there exists $t\in \mathcal{F}_-^u(p_0)$ such that $\mathcal{F}_-^u(p_0)\cap \mathfrak{R}=[p_0,t)^u$. Even more, since the interior of $\mathfrak{R}$ is trivially bifoliated, there exists $f:(x,x_0]^u\rightarrow (t,p_0]^u$ a homeomorphism such that for every $z\in (x,x_0]^u$ $$\mathcal{F}^s(z)\cap \mathcal{F}_-^u(p_0)=\{f(z)\}$$

Since $l$ does not intersect $\mathcal{F}^u(p_0)$, we get that $\mathcal{F}^s(t)$ is not separated from $l$. Moreover, the transversality of $\mathcal{F}^s$ and $\mathcal{F}^u$, together with the definition of $p_0$ and the fact that $\mathfrak{R}$ does not contain any singularity of $\mathcal{F}^{s,u}$ imply that for any point $y\in [x_0,p_0]^s$ that is sufficiently close to $p_0$ $$\mathcal{F}^u(y)\cap \mathcal{F}^s(t)\neq \emptyset \text{ and }\mathcal{F}^u(y)\cap l\neq \emptyset$$ This contradicts Item (4) of Lemma \ref{l.nonseparationlemma} applied for the leaves $l$ and $\mathcal{F}^s(t)$. It follows that $\mathcal{F}_-^u(p_0)\cap \mathfrak{R}=\mathcal{F}_-^u(p_0)$, which proves the desired result. 
\end{proof}

\vspace{0.5cm}
\textbf{Claim 3.} The leaf $\mathcal{F}^u(p_0)$ is periodic for $\rho$. 
\vspace{0.2cm}

\noindent \textit{Proof of Claim 3.}
%We will prove that $\mathcal{F}^u(p_0)$ is periodic. The fact that $\mathcal{F}^u(q_0)$ is also periodic will follow from a similar argument. 
If $p_0$ belongs in the $\mathcal{F}^u$-boundary of a rectangle in $\mathcal{R}$, then we deduce the periodicity of $\mathcal{F}^u(p_0)$  from Lemma \ref{l.boundaryleavesareperiodic}. We will therefore assume from now on and for the rest of the proof of this claim that $p_0$ does not belong in the $\mathcal{F}^u$-boundary of any rectangle in $\mathcal{R}$. 

Denote by $\mathcal{F}_-^u(p_0)$ the $\mathcal{F}^u$ separatrix of $p_0$ contained in $\mathfrak{R}$. Thanks to the strong finite return time axiom, there exists $$R_0,...,R_n...$$ a sequence of rectangles in $\mathcal{R}$ such that $p_0\in R_k$, $R_k\cap \inte{\mathfrak{R}}\neq\emptyset$, and $R_{k+1}\cap R_k$ is a vertical subrectangle of $R_k$ for every $k\in\mathbb{N}$ (see Figure \ref{f.nonseparated3}). Notice that, since $\mathfrak{R}$ is trivially bifoliated, Proposition \ref{p.propertiesoffoliformarkovianactions} implies that $R_k$ admits at most one $\mathcal{F}^s$-boundary component in the interior of $\mathfrak{R}$. Moreover, by Claim 2 and the expansivity axiom, after possibly passing to a subsequence of $(R_n)_{n\in\mathbb{N}}$, we may assume that for every $k\in\mathbb{N}$ the rectangle $R_k$ admits exactly one $\mathcal{F}^s$-boundary component in the interior of $\mathfrak{R}$.

  \begin{figure}[h!]
    \centering
    \includegraphics[width=0.5\linewidth]{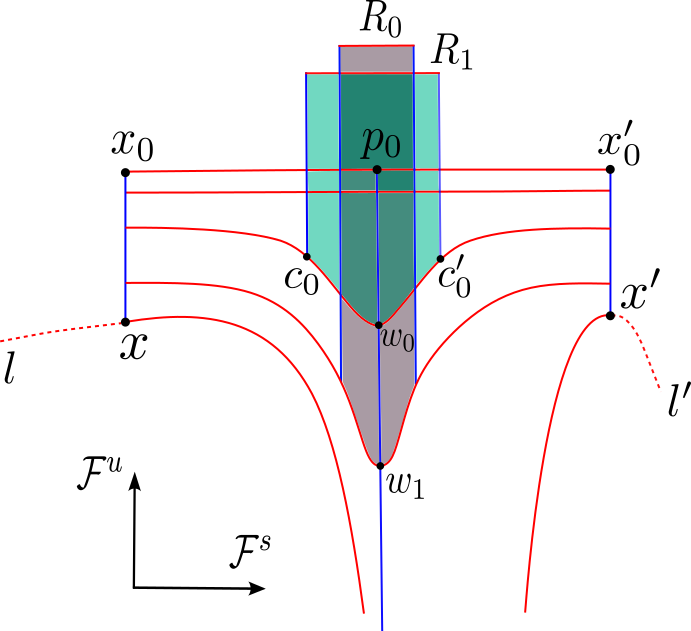}
    \caption{The bottom-left corner of the figure shows our chosen orientations of $\mathcal{F}^s$ and $\mathcal{F}^u$.}
    \label{f.nonseparated3}
\end{figure}

Endow the $\mathcal{F}^s$ and the $\mathcal{F}^u$ leaves in $\mathfrak{R}$ with orientations so that any  $\mathcal{F}^s$-segment going from $(x,x_0]^u$ to $(x',x_0']^u$ becomes positively oriented and any $\mathcal{F}^u$-segment going from $[x_0,x_0']^s$ to $l$ or $l'$ becomes negatively oriented. Thanks to the previous choices of orientations, we can naturally orient the $\mathcal{F}^s$ and $\mathcal{F}^u$ leaves inside every $R_k$. 

By the finiteness axiom, after possibly  considering a subsequence of $(R_n)_{n\in \mathbb{N}}$, we can assume without any loss of generality that all the rectangles in $(R_n)_{n\in \mathbb{N}}$ belong in the same $\rho$-orbit of rectangles. Denote by $g_k\in G$ the unique element in $G$ (see Lemma \ref{l.preserverectangleimpliestorsion}) for which $\rho(g_k)(R_k)=R_{k+1}$. Passing to a further subsequence if necessary, we will assume without any loss of generality that $\rho(g_k)$ sends the oriented $\mathcal{F}^{s,u}$-leaves of $R_k$ to the oriented $\mathcal{F}^{s,u}$-leaves of $R_{k+1}$. 

Under the previous assumptions, we will show that $\rho(g_0)$ preserves $\mathcal{F}^u(p_0)$. Indeed, denote by $\partial_-^sR_k$ the unique $\mathcal{F}^s$-boundary component of $R_k$ that lies in the interior of $\mathfrak{R}$ and by $w_k$ 
the unique point inside $\mathcal{F}_-^u(p_0)\cap \partial_-^sR_k$ (see Figure \ref{f.nonseparated3}). By construction, $\rho(g_0)(\partial_-^sR_0)=\partial_-^sR_1$; hence 
\begin{equation*}
    \rho(g_0)(w_0)\in \partial_-^sR_1\subset \inte{\mathfrak{R}}
\end{equation*}

If $w_1= \rho(g_0)(w_0)$, then $\rho(g_0)$ preserves $\mathcal{F}^u(p_0)$, which proves the desired result. Suppose from now on that $$w_1\neq \rho(g_0)(w_0)$$
In order to prove that this is impossible, we consider the two following cases.

\textbf{Case 1: The segment $[w_1,\rho(g_0)(w_0)]^s$ is positively oriented}

In this case, take $c_0$ to be the extremity of $\partial^s_-R_0$ for which $[c_0,w_0]^s$ is positively oriented (see Figure \ref{f.nonseparated3}). Recall that $p_0$ and thus also $w_0$ is not contained in the $\mathcal{F}^u$-boundary of $R_0$; therefore, $c_0\neq w_0$. By Claim 1 and since the interior of $\mathfrak{R}$ is trivially bifoliated, we have that $w_0$ is the only point in $[c_0,w_0]^s$ whose $\mathcal{F}^u$-leaf does not intersect $l$. More precisely, for every $z\in [c_0,w_0)^s$ the negative $\mathcal{F}^u$-separatrix of $z$ intersects $l$. 

Recall that $\rho(g_0)$ sends oriented $\mathcal{F}^{s,u}$-segments of $R_0$ to oriented $\mathcal{F}^{s,u}$-segments of $R_1$ and also that $p_0$ does not belong to the $\mathcal{F}^u$-boundary of any rectangle in $\mathcal{R}$. Using the previous facts, as well as the fact that $[w_1,\rho(g_0)(w_0)]^s$ is non-trivial and positively oriented, we get that $w_1$ lies in the interior of $\rho(g_0)([c_0,w_0]^s)$ and also that the negative $\mathcal{F}^u$-separatrix of $w_1$ intersects $\rho(g_0)(l)$. By Claim 2, this implies that $\rho(g_0)(l)$ intersects the interior of $\mathfrak{R}$ and thus intersects the $\mathcal{F}^u$-leaf of $\rho(g_0)(w_0)\in \inte{\mathfrak{R}}$. This is impossible as $\mathcal{F}^u(p_0)=\mathcal{F}^u(w_0)$ does not intersect $l$. 

\vspace{0.3cm}
\textbf{Case 2: The segment $[w_1,\rho(g_0)(w_0)]^s$ is negatively oriented}

In this case, take $c'_0$ to be the extremity of $\partial^s_-R_0$ for which $[w_0,c'_0]^s$ is positively oriented (see Figure \ref{f.nonseparated3}). Recall that $p_0$ and thus also $w_0$ is not contained in the $\mathcal{F}^u$-boundary of $R_0$; therefore, $c'_0\neq w_0$. 

Similarly to Case 1, we have that $w_1$ lies in the interior of $\rho(g_0)([w_0, c_0']^s)$. Therefore, thanks to Claims 1 and 2, the negative $\mathcal{F}^u$-separatrix of $\rho(g_0)(w_0)$ intersects $l$ and $l$ intersects the interior of $\rho(g_0)(\mathfrak{R})$. Since $\rho(g_0)(\mathfrak{R})$ is trivially bifoliated and $w_1\in \rho(g_0)([w_0, c_0']^s)\subset \rho(g_0)(\inte{\mathfrak{R}})$, we get that $l\cap \mathcal{F}^u(w_1)\neq \emptyset$, which is absurd, as $\mathcal{F}^u(p_0)=\mathcal{F}^u(w_1)$ does not intersect $l$.  

\vspace{0.2cm}
Our previous arguments imply that $\rho(g_0)(w_0)=w_1$ and thus that $\rho(g_0)$ preserves $\mathcal{F}^u(p_0)$, which finishes the proof of Claim 3. \qed

\vspace{0.4cm}
We are now ready to show that $l$ is periodic, thereby completing the argument for Lemma \ref{l.nonseparatedimpliesfixed}. Take $g\in \text{Stab}_{\rho}(\mathcal{F}^u(p_0))-\{e\}$ such that $\rho(g)(\mathcal{F}_-^u(p_0))\subseteq \mathcal{F}_-^u(p_0)$. Such a $g$ exists, thanks to Claim 2 and Lemma \ref{l.hyperbolicityforperiodicpoints}. By the definition of $g$, we have that $$\underset{y\in \rho(g)(\mathcal{F}_-^u(p_0))}{\bigcup}\mathcal{F}^s(y)\subset \underset{y\in \mathcal{F}_-^u(p_0)}{\bigcup}\mathcal{F}^s(y)$$ This implies that the $\mathcal{F}^s$-saturated neighborhoods of $l$ and $\rho(g)(l)$ always intersect along some leaf in $\underset{y\in \mathcal{F}_-^u(p_0)}{\bigcup}\mathcal{F}^s(y)$; thus  $\rho(g)(l)=l$ or $\rho(g)(l)$ is non separated from $l$. 

Next, observe that the facts that $\rho(g)$ is orientation preserving and that $[x_0,p_0]^s$ is positively oriented imply that the segment $[\rho(g)(x_0),\rho(g)(p_0)]^s$ is also positively oriented. Consequently, for every point $z\in[x_0,p_0)^s$ sufficiently close to $p_0$, the following hold:

\begin{enumerate}
     \vspace{0.1cm}
    \item by Claim 1, $\mathcal{F}^u(z)\cap l\neq\emptyset$, and therefore $\mathcal{F}^u(\rho(g)(z))\cap \rho(g)(l)\neq\emptyset$
    
    \vspace{0.1cm}
    \item since $\rho(g)(p_0)\in \mathcal{F}^u_-(p_0)\subset \mathfrak{R}$ (the inclusion follows from Claim 2), we have $\rho(g)(z)\in \mathfrak{R}$. Moreover, since $\mathfrak{R}$ is trivially bifoliated, $\mathcal{F}^u(\rho(g)(z))\cap [x_0,p_0)^s\neq\emptyset$

     \vspace{0.1cm}
    \item since $\mathcal{F}^u(\rho(g)(z))\cap [x_0,p_0)^s\neq\emptyset$, Claim 1 implies that $\mathcal{F}^u(\rho(g)(z))\cap l\neq\emptyset$
\end{enumerate}

The fact that $\mathcal{F}^u(\rho(g)(z))$ intersects both $l$ and $\rho(g)(l)$ implies, by Item (4) of Lemma \ref{l.nonseparationlemma}, that $\rho(g)(l)$ cannot be distinct from $l$ and non-separated from it. Hence, $\rho(g)(l)=l$, which finishes the proof of the lemma. 

\end{proof}
We remark that the argument used in the final part of the proof of Lemma~\ref{l.nonseparatedimpliesfixed} also yields the following statement:

\begin{lemm}\label{l.perfectfitperiodic}
    Consider $l^s\in \mathcal{F}^s$ and $l^u\in \mathcal{F}^u$ two leaves forming a perfect fit. Assume that $l^s$ is periodic for $\rho$ and let $p^s$ be the unique periodic point contained in $l^s$. For any element $g\in \text{Stab}_{\rho}(p^s)$ such that $\rho(g)$ preserves every $\mathcal{F}^{s,u}$-separatrix of $p^s$, we have that $\rho(g)$ preserves both $l^s$ and $l^u$. 
\end{lemm}

Finally, before concluding this section, we state the following generalization of Lemma \ref{l.nonseparatedimpliesfixed}, which will be used during the next section. 

\begin{lemm}\label{l.fixatthesametimenonsepleaves}
    Consider $S$ a set of non-separated $\mathcal{F}^s$- or $\mathcal{F}^u$-leaves in $\mathcal{P}$. We have that there exists $g\in G-\{e\}$ such that $\rho(g)(l)=l$ for every $l\in S$. 
\end{lemm}
The previous lemma is proved in \cite{nontransitiveanosovlike} for actions satisfying Axioms~(A1)--(A4) of Definition~\ref{d.anosovlike} (see Proposition~2.12). Since $\rho$ has been shown to satisfy these axioms, we refer to \cite{nontransitiveanosovlike} for the proof.

\section{On the non existence of ideal quadrilaterals in $\mathcal{P}$}
In this final section of this paper, we prove that $\rho$ satisfies Axiom (A5) of Definition \ref{d.anosovlike}, thereby completing the proof of Theorem \ref{t.main}. The proof of this fact relies on the Lemmas \ref{l.embeddingweirdrectangle} and \ref{l.noinfiniteproductregion}, which restrict the behavior of $\mathcal{F}^s$ and $\mathcal{F}^u$ inside $\mathcal{P}$.

\begin{lemm}\label{l.embeddingweirdrectangle}
    There does not exist an embedding $\phi: (0,1)^2\rightarrow \mathcal{P}$ such that: 

    \begin{enumerate}
        \item $\phi((0,1) \times \{t\})$ is an $\mathcal{F}^s$-segment in $\mathcal{P}$ for every $t\in (0,1)$ 

        \vspace{0.1cm}
        \item $\phi(\{t\}\times (0,1))$ is an $\mathcal{F}^u$-segment in  $\mathcal{P}$ for every $t\in (0,1)$ 
        \vspace{0.1cm}
        \item as $t\to 1$, the accumulation set of the family of $\mathcal{F}^s$-segments $\phi((0,1)\times\{t\})$ is exactly one face of a single $\mathcal{F}^s$-leaf
        \vspace{0.1cm}
        \item as $t\to 1$, the accumulation set of the family of  $\mathcal{F}^u$-segments $\phi(\{t\}\times (0,1))$ is exactly a union of faces contained in a family of at least two non-separated $\mathcal{F}^u$-leaves
    \end{enumerate}
\end{lemm}

Naturally, the same result holds if we change the roles of $\mathcal{F}^s$ and $\mathcal{F}^u$. In simpler terms, the above lemma excludes the existence of trivially bifoliated regions bounded by above or below by a single $\mathcal{F}^s$-leaf and on one side by several non-separated $\mathcal{F}^u$-leaves (see Figures \ref{f.case1proof} and \ref{f.case2proof}).
\begin{proof}
    Assume by contradiction that such an embedding $\phi$ exists. Let $V:=\phi\big( (0,1)^2\big)$, $L^s$ be the $\mathcal{F}^s$-leaf on which the segments $\phi((0,1) \times \{t\})$ accumulate when $t\rightarrow 1$ and $\mathcal{L}^u$ be the set of $\mathcal{F}^u$-leaves on which the segments $\phi(\{t\}\times (0,1))$ accumulate when $t\rightarrow 1$.
    
    By our definition of $\phi$, the set $V$ is trivially bifoliated by $\mathcal{F}^s$ and $\mathcal{F}^u$. Also, by Lemmas \ref{l.periodicleavescarryuniqueperiodicpoint} and \ref{l.nonseparatedimpliesfixed}, every $L\in \mathcal{L}^u$ carries a unique periodic point for $\rho$, which will be denoted by $p_L$. We remark that, since $\partial V$ contains a face of each leaf in $\mathcal{L}^u$, we have that $p_L\in \partial V$ for every $L\in \mathcal{L}^u$. In particular, for every $L\in \mathcal{L}^u$ there exists a unique $\mathcal{F}^s$-separatrix of $p_L$ intersecting $V$, which we will denote by $\mathcal{F}^s_V(p_L)$ (the uniqueness of the previous separatrix follows from Proposition \ref{p.propertiesoffoliformarkovianactions} and from the fact that $V$ is trivially bifoliated).

   Thanks to  Lemma \ref{l.fixatthesametimenonsepleaves}, there exists $h\in G-\{e\}$ such that $\rho(h)$ preserves every leaf in $\mathcal{L}^u$. By possibly replacing $h$ by a suitable power $h^k$, with $k\in \mathbb{N}$, we may assume that $\rho(h)$ preserves each $\mathcal{F}^{s,u}$-separatrix of $p_{L_0}$ for some $L_0\in \mathcal{L}^u$. Since $V$ is trivially bifoliated, it follows that $\rho(h)$ preserves each $\mathcal{F}^{s,u}$-separatrix of $p_L$ for every $L\in \mathcal{L}^u$ (see Figure \ref{f.case2proof}). Furthermore, after possibly replacing $h$ by $h^{-1}$, we may assume that $\rho(h)$ acts as a contraction on $\mathcal{F}^s_V(p_{L_0})$. The trivial bifoliation of $V$ then implies that $\rho(h)$ acts as a contraction on $\mathcal{F}^s_V(p_L)$ for every $L\in \mathcal{L}^u$.

   Orient the $\mathcal{F}^u$-leaves inside $V$ so that any $\mathcal{F}^u$-segment going from $V$ to $L^s$ is positively oriented. This orientation induces a total order on $\mathcal{L}^u$. We distinguish the following two cases.

\vspace{0.3cm}
\textbf{Case 1: There exists a maximal element in $\mathcal{L}^u$}

Suppose that $\mathcal{L}^u$ admits a maximal element, say $L_{max}$. In this case, using the definition of $\phi$, it is not difficult to see that $L_{max}$ forms a perfect fit with $L^s$ (see Figure \ref{f.case1proof}). Hence, $\rho(h)$ preserves $L_{max}$ and $L^s$, and fixes both $p_{L_{max}}$ and $p_{L^s}$, where $p_{L^s}$ denotes the unique periodic point of $\rho$ in $L^s$. Since $\partial V$ contains a face of $L^s$, we have that $\mathcal{F}^u(p_{L^s})$ intersects $V$. After possibly replacing $h$ by some suitable power $h^k$, with $k\in \mathbb{N}$, assume without any loss of generality that $\rho(h)$ preserves the unique $\mathcal{F}^u$-separatrix of $p_{L^s}$ that intersects $V$. Since $V$ is trivially bifoliated, we get that, in addition to $p_{L_{max}}$ and  $p_{L^s}$, the homeomorphism $\rho(h)$ fixes the unique point of intersection of $\mathcal{F}^u(p_{L^s})$ and  $\mathcal{F}^s_V(p_{L_{max}})$, which contradicts Lemma \ref{l.periodicleavescarryuniqueperiodicpoint}.

 \begin{figure}[h!]
    \centering
    \includegraphics[scale=0.4]{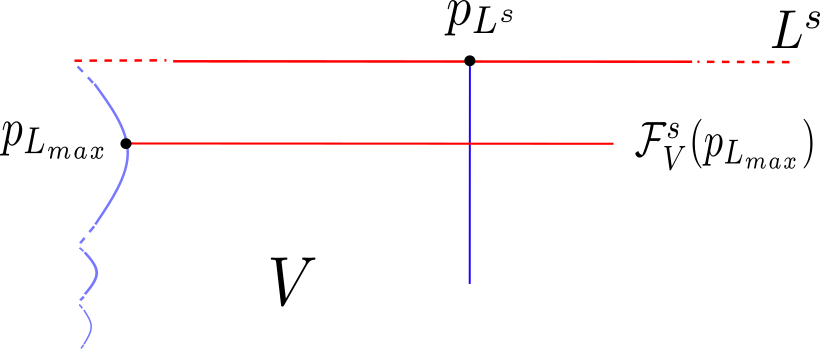}
    \caption{If $\mathcal{L}^u$ admits a maximal element, then $\rho(h)$ admits multiple fixed points on the same $\mathcal{F}^{s,u}$-leaf}
    \label{f.case1proof}
\end{figure}

\vspace{0.3cm}
\textbf{Case 2: There is no maximal element in $\mathcal{L}^u$}

In this case, thanks to the definition of $\phi$, there exists a sequence $(L_n)_{n\in \mathbb{N}}$ in  $\mathcal{L}^u$ such that the accumulation set of $(\mathcal{F}^s_V(p_{L_n}))_{n\in\mathbb{N}}$ contains $L^s$ (see Figure \ref{f.case2proof}). Since $\rho(h)$ preserves every separatrix in $(\mathcal{F}^s_V(p_{L_n}))_{n\in\mathbb{N}}$, it follows that $\rho(h)(L^s)=L^s$ or that $\rho(h)(L^s)$ is non-separated from $L^s$.

 \begin{figure}[h!]
    \centering
    \includegraphics[scale=0.35]{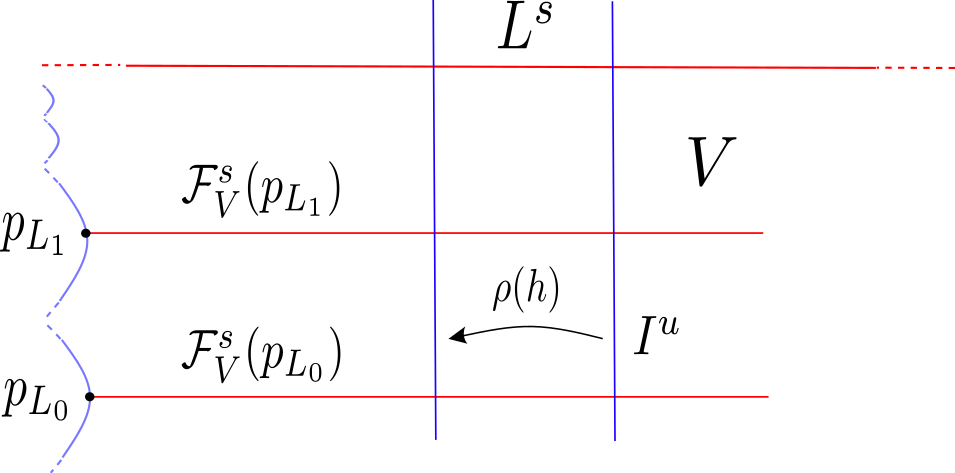}
    \caption{If $\mathcal{L}^u$ does not admit a maximal element, then $\rho(h)$ admits multiple fixed points on the same $\mathcal{F}^{s,u}$-leaf}
    \label{f.case2proof}
\end{figure}

We now show that the second possibility cannot occur. Indeed, let $I^u$ be any $\mathcal{F}^u$-leaf intersecting both $L^s$ and $\mathcal{F}^s_V(p_{L_0})$ (such a leaf exists, since $V$ is trivially bifoliated). On the one hand, since $\rho(h)$ acts as a contraction on $\mathcal{F}^s_V(p_{L_0})$ and $V$ is trivially bifoliated, we have that $\rho(h)(I^u)$ is an $\mathcal{F}^u$-leaf intersecting both $\mathcal{F}^s_V(p_{L_0})$ and $L^s$. On the other hand, since $I^u$ intersects $L^s$, the leaf $\rho(h)(I^u)$ must also intersect $\rho(h)(L^s)$. By Item (4) of Lemma \ref{l.nonseparationlemma}, we deduce that $\rho(h)(L^s)$ cannot be distinct from $L^s$ and non-separated from it. Hence,
$$
\rho(h)(L^s)=L^s
$$

In particular, $\rho(h)$ fixes both $p_{L_0}$ and $p_{L^s}$, the unique periodic point for $\rho$ in $L^s$. Similarly to Case 1, one can show that this implies the existence of two periodic points in $\mathcal{F}^s_V(p_{L_0})$, which contradicts Lemma \ref{l.periodicleavescarryuniqueperiodicpoint}.
\end{proof}

\begin{defi}\label{d.infiniteproduct}
Consider $U$ a subset of $\mathcal{P}$ that is homeomorphic to $[0,1]\times \mathbb{R}$. We will say that $U$ is an \emph{infinite product region} if, up to changing the roles of $\mathcal{F}^s$ and $\mathcal{F}^u$, there exists a homeomorphism $h:U\rightarrow [0,1]\times \mathbb{R}$ sending every $\mathcal{F}^s$-segment (resp. $\mathcal{F}^u$-segment) in $U$ to a horizontal (resp. vertical) segment in $[0,1]\times \mathbb{R}$.
\end{defi}

\begin{lemm}\label{l.noinfiniteproductregion}
    The existence of an infinite product region in $(\mathcal{P},\mathcal{F}^s,\mathcal{F}^u)$ implies that $(\mathcal{P},\mathcal{F}^s,\mathcal{F}^u)$ is trivially bifoliated (i.e. every leaf of $\mathcal{F}^s$ intersects every leaf of $\mathcal{F}^u$). 
\end{lemm}

The above lemma was originally proved in \cite{Fe} for the natural action of $\pi_1(M)$ on the bifoliated plane of an Anosov flow $(M^3,\Phi)$. Later this was generalized for Anosov-like actions in \cite{circleatinfinity}. Our proof below follows the general strategy of the proof in \cite{Fe}. 

\begin{proof}
    Consider $U$ an infinite product region in $(\mathcal{P},\mathcal{F}^s,\mathcal{F}^u)$. Assume, without loss of generality, that $U$ is bounded by two closed, unbounded $\mathcal{F}^u$-segments $L_1$ and $L_2$ with respective endpoints $x_1$ and $x_2$, and by the $\mathcal{F}^s$-segment $[x_1,x_2]^s$ (see Figure \ref{f.replacingU}). Using the existence of $U$, we begin the proof of this lemma by constructing an infinite product region whose $\mathcal{F}^s$-boundary contains a periodic point.

  \begin{figure}[h!]
    \centering
    \includegraphics[width=0.5\linewidth]{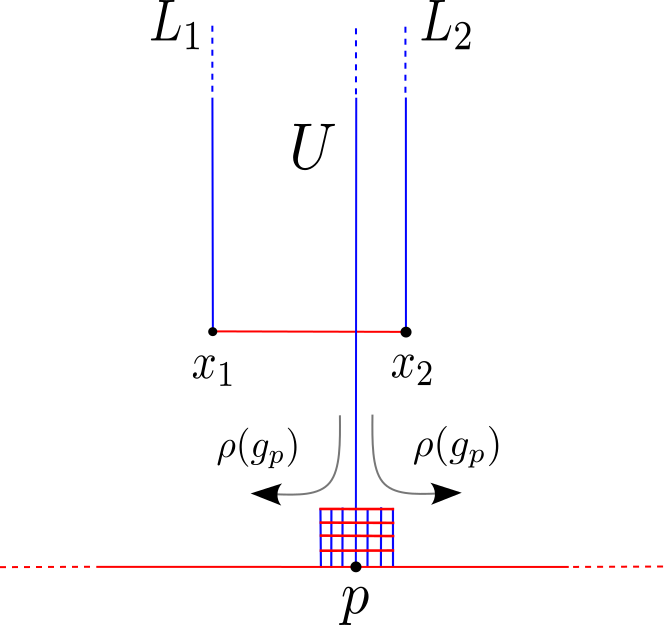}
    \caption{Replacing $U$ by an infinite product region bounded by two $\mathcal{F}^u$ half-leaves and an $\mathcal{F}^s$-segment containing a periodic point for $\rho$.}
    \label{f.replacingU}
\end{figure}

    \vspace{0.5cm}

    \textbf{Claim 1.} There exists an infinite product region in $(\mathcal{F}^s,\mathcal{F}^u)$ bounded by two closed, unbounded $\mathcal{F}^u$-segments and an $\mathcal{F}^s$-segment whose interior contains a periodic point $p$.
    \begin{proof}[Proof of the Claim]
    
  First, note that as an immediate consequence of Definition \ref{d.infiniteproduct},  $\rho(g)(U)$ is an infinite product region for every $g\in G$. Next, thanks to Lemmas \ref{l.boundaryleavesareperiodic} and \ref{l.boundaryleavesarefiniteanddense}, there exists $l$ a periodic $\mathcal{F}^u$-leaf for $\rho$ intersecting the interior of $U$. Denote by $p$ the unique periodic point  inside $l$ (see Lemma \ref{l.periodicleavescarryuniqueperiodicpoint}). 
    
    If $p\in U$, after cutting $U$ along $\mathcal{F}^s(p)$, one obtains an infinite product region with the desired properties. If $p\notin U$, then take $g_p\in \text{Stab}_{\rho}(p)$ such that $\rho(g_p)$ acts as a contraction on the $\mathcal{F}^u$-separatrix of $p$ that intersects $U$ (the existence of $g_p$ follows from Lemma \ref{l.hyperbolicityforperiodicpoints}). Recall now that the closure of any connected component of $\mathcal{P}-\mathcal{F}^s(p)$ contains trivially bifoliated neighborhoods of $p$ (see Figure \ref{f.replacingU}). By the definition of $g_p$, there exists $N\in\mathbb{N}$ sufficiently large such that $\rho(g_p^N)(U)$ intersects one of these neighborhoods. The union of this neighborhood with $\rho(g_p^N)(U)$  contains an infinite product region with the desired properties.

    \end{proof}

    Thanks to Claim 1, we may assume without any loss of generality that $(x_1,x_2)^s$ contains a periodic point $p$. Consider $g_p\in \text{Stab}_{\rho}(p)$ such that $\rho(g_p)$ acts as a contraction on every $\mathcal{F}^u$-separatrix of $p$ and as an expansion on every $\mathcal{F}^s$-separatrix of $p$. Let $$V:= \underset{n\in \mathbb{N}}{\bigcup}\rho(g_p^n)(U)$$ Since $U$ is an infinite product region and $p\in (x_1,x_2)^s$, it is not difficult to see that $V$ is trivially bifoliated (see Figure \ref{f.Vboundedbyinfiniteleaves}). Moreover, thanks to Lemma \ref{l.hyperbolicityforperiodicpoints}, we have that $V$ contains a face of $\mathcal{F}^s(p)$, which we will denote by $\partial^sV$. 

\vspace{0.5cm}

\textbf{Claim 2.} $V$ coincides with the closure a connected component of $\mathcal{P}-\mathcal{F}^s(p)$
\begin{proof}[Proof of the Claim]
 Let us prove that the sequence of leaves $(\rho(g_p^{n})(L_1))_{n\in \mathbb{N}}$ exits every compact set in $\mathcal{P}$. Assume by contradiction that this is false. We distinguish two cases.

\vspace{0.3cm}
\textbf{Case 1: The sequence $(\rho(g_p^{n})(L_1))_{n\in \mathbb{N}}$ accumulates on a face of a single $\mathcal{F}^u$-leaf, denoted by $l'$}

In this case, $l'$ is preserved by $\rho(g_p)$ and therefore admits a periodic point $p'$ fixed by $\rho(g_p)$ (see Lemma \ref{l.periodicleavescarryuniqueperiodicpoint}). As $p'$ belongs to every face of $l'$, $\mathcal{F}^s(p')$ intersects the interior of $V$. Using the facts that $p$ and $p'$ are fixed by $\rho(g_p)$ and that $V$ is trivially bifoliated, we conclude that the unique point of $\mathcal{F}^s(p')\cap \mathcal{F}^u(p)$ is also fixed by $\rho(g_p)$, which contradicts Lemma \ref{l.periodicleavescarryuniqueperiodicpoint}.

\vspace{0.3cm}
\textbf{Case 2: The sequence $(\rho(g_p^{n})(L_1))_{n\in \mathbb{N}}$ accumulates on a union of faces contained in a family of at least two non-separated $\mathcal{F}^u$-leaves}

In this case, $V$ is the image of an embedding satisfying the hypotheses of Lemma \ref{l.embeddingweirdrectangle} (see Figure \ref{f.Vboundedbyinfiniteleaves}); hence, this case is also impossible, which concludes the proof of the claim.

 \begin{figure}[h!]
    \centering
    \includegraphics[width=0.6\linewidth]{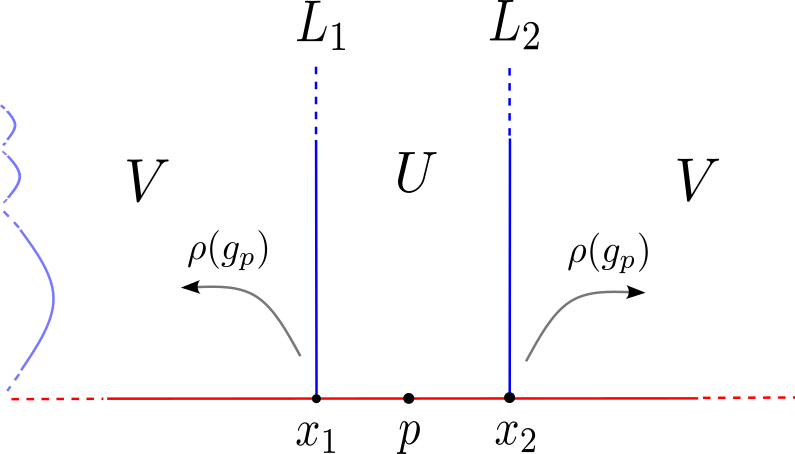}
    \caption{Extending $U$ into a larger trivially bifoliated region.}
    \label{f.Vboundedbyinfiniteleaves}
\end{figure}

\end{proof}

Let us now conclude the proof of the lemma. By Claim 2, $V$ is a trivially bifoliated connected component of $\mathcal{P}-\mathcal{F}^s(p)$. Thanks to Lemmas \ref{l.boundaryleavesareperiodic} and \ref{l.boundaryleavesarefiniteanddense}, there exists a periodic $\mathcal{F}^s$-leaf that is entirely contained in $V$. Denote by $p'$ its unique periodic point and by $g_{p'}$ an element in $\text{Stab}_{\rho}(p')$ acting as an expansion on $\mathcal{F}^u(p')$. 

Assume that $p$ is a singular periodic point. If this were true, then $p$ would be the only singular periodic point in $V$ and  therefore, $\rho(g_{p'})$ would fix $p'$, $p$ and also the unique intersection point of $\mathcal{F}^u(p')$ and $\partial^sV$. This contradicts Lemma \ref{l.periodicleavescarryuniqueperiodicpoint} and implies that $p$ is a regular periodic point for which $\partial^sV=\mathcal{F}^s(p)$. Finally, consider $W:=\underset{n\in \mathbb{N}}{\bigcup}\rho(g_{p'}^{n})(V)$. By adapting our proof of Claim 2, one can show that $W$ is trivially bifoliated and coincides with $\mathcal{P}$, which finishes the proof of the lemma. 
\end{proof}

We are now ready to prove the main result of this section. 

\begin{prop}
    There are no totally ideal quadrilaterals inside $(\mathcal{P},\mathcal{F}^s,\mathcal{F}^u)$.
\end{prop}
The above proposition was originally proved in \cite{Fe1} for the natural action of $\pi_1(M)$ on the bifoliated plane of an Anosov flow $(M^3,\Phi)$. Our proof below follows the general strategy of the proof in \cite{Fe1}.

\begin{proof}
    Assume by contradiction that there exists a totally ideal quadrilateral $U$ inside $(\mathcal{P},\mathcal{F}^s,\mathcal{F}^u)$. Recall that by Definition \ref{d.ideal}, $U$ is an open subset of $\mathcal{P}$ for which there exist two $\mathcal{F}^s$-leaves, say $l_-^s$ and $l_+^s$, and two $\mathcal{F}^u$-leaves, say $l_-^u$ and $l_+^u$, such that:
    
    \begin{itemize}
        \item the boundary of $U$ is the disjoint union of four faces contained in $l_-^s$, $l_+^s$, $l_-^u$, and $l_+^u$ respectively (see Figure \ref{f.idealrectangles})
        \item $l^s_+$ and  $l^s_-$ form perfect fits with $l^u_+$ and $l^u_-$
    \end{itemize}
    
   \noindent We remark also that thanks to the third axiom of Definition \ref{d.ideal}, $U$ is trivially bifoliated by $\mathcal{F}^s$ and $\mathcal{F}^u$. 

\vspace{0.2cm}

\noindent \textit{Totally ideal quadrilaterals can not be periodic}
\vspace{0.2cm}

Using the fact that $U$ is trivially bifoliated, let us prove that:

\vspace{0.2cm}
\textbf{Claim 1.} None of the leaves $l_-^s$, $l_+^s$, $l_-^u$ and $l_+^u$ can be periodic for $\rho$. 
\begin{proof}[Proof of Claim 1]
   Assume, for contradiction and without loss of generality, that $l_-^s$ is periodic under $\rho$. In this case, $l_-^s$ admits a unique periodic point $p^s$. Moreover, since the boundary of $U$ contains a face of $l_-^s$, we have that $p^s\in \partial U$; hence, $\mathcal{F}^u(p^s)$ intersects $U$. 
   
   Consider now $g_{p^s}\in \text{Stab}_{\rho}(p^s)-\{e\}$ such that $\rho(g_{p^s})$ preserves each $\mathcal{F}^{s,u}$-separatrix of $p^s$. Using Lemma \ref{l.perfectfitperiodic} as well as the fact that $l^u_-$ and $l^s_-$ form a perfect fit, we get that $\rho(g_{p^s})$ preserves both  $l^s_-$ and $l^u_-$. Therefore, $\rho(g_{p^s})$ fixes a unique periodic point in $l^u_-$, say $p^u$. Once again, since the boundary of $U$ contains a face of $l_-^u$, we have that $p^u\in \partial U$ and also that $\mathcal{F}^s(p^u)$ intersects $U$. Finally, since $U$ is trivially bifoliated, in addition to $p^s$ and $p^u$, $\rho(g_{p^s})$ fixes the unique point of intersection of $\mathcal{F}^u(p^s)$ and $\mathcal{F}^s(p^u)$, which contradicts Lemma \ref{l.periodicleavescarryuniqueperiodicpoint}.
   
%and acts as a topological expansion (resp. contraction) on each $\mathcal{F}^s$-separatrix (resp. $\mathcal{F}^u$-separatrix) of $p^s$. 

%Consider $x\in l_-^s$ such that $\mathcal{F}^u(x)$ crosses $U$ and is very close to $l^u_-$ (such a point exists, because $l^s_-$ and $l^u_-$ form a perfect fit) and $y\in \mathcal{F}^u(p^s)\cap U$. By our definition of $g_{p^s}$ and since $U$ is a totally ideal quadrilateral, we have that 
%   \begin{enumerate}
%\item the sequence $\big(\rho(g_{p^s}^k)(\mathcal{F}^u(x)\big)_{k\in\mathbb{N}}$ consists of $\mathcal{F}^u$-leaves crossing $U$, and its accumulation set contains both $l^u_-$ and $\rho(g_{p^s})(l^u_-)$. We deduce that, either $\rho(g_{p^s})(l^u_-)=l^u_-$ or $\rho(g_{p^s})(l^u_-)$ is non-separated from $l^u_-$

%\vspace{0.2cm}

%\item $\rho(g_{p^s})(\mathcal{F}^s(y))$ intersects both $l^u_-$ and $\rho(g_{p^s})(l^u_-)$
%   \end{enumerate}

\end{proof}

Thanks to Claim 1, the leaves $l_-^s$, $l_+^s$, $l_-^u$ and $l_+^u$ are all non-periodic. Therefore, by Lemma \ref{l.singularareperiodic}, the leaves $l_-^s$ and $l_+^s$ (resp. $l_-^u$ and $l_+^u$) are regular $\mathcal{F}^s$-leaves (resp. $\mathcal{F}^u$-leaves) and thus $$ \partial U = l_-^s\cup l_+^s\cup l_-^u \cup l_+^u$$ Also, by Lemma \ref{l.boundaryleavesareperiodic}, none of the leaves $l_-^s$, $l_+^s$, $l_-^u$, and $l_+^u$ can be a boundary leaf of $\mathcal{R}$ (see Definition \ref{d.boundaryleaf}). 

\vspace{0.2cm}

\noindent \textit{Constructing a sequence of totally ideal quadrilaterals}
\vspace{0.2cm}

Consider $x\in l_-^s$. By a repeated application of the strong finite return time axiom, there exists $$R_0,...,R_n...$$ a sequence of rectangles in $\mathcal{R}$ such that $x\in R_k$ and $R_{k+1}\cap R_k$ is a horizontal subrectangle of $R_k$ for every $k\in \mathbb{N}$ (see Figure \ref{f.idealrectangles}).

\begin{figure}[h!]
    \centering
    \includegraphics[width=0.45\linewidth]{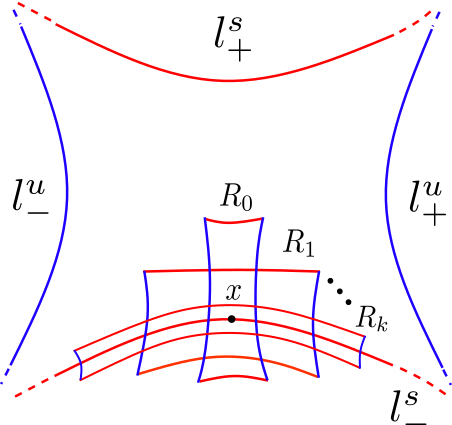}
    \caption{Generating a sequence of totally ideal quadrilaterals.}
    \label{f.idealrectangles}
\end{figure}

Thanks to the finiteness and the expansivity axioms, after possibly considering a subsequence of $(R_n)_{n\in \mathbb{N}}$, we may assume that the rectangles $R_0,...,R_n...$ belong in the same $\rho$-orbit and that $R_k\cap \partial U\subset l_-^s$ for every $k\in \mathbb{N}$. In particular, for every $k\in\mathbb{N}$, the rectangle $R_k$ has a unique $\mathcal{F}^s$-boundary component contained in $U$ and, for every $k,l\in\mathbb{N}$, there exists $g_{k,l}\in G$ such that
\begin{equation}\label{eq.definitiongkl}
    R_l=\rho(g_{k,l})(R_k)
\end{equation}

Notice that $g_{k,l}$ is uniquely defined by the above property, thanks to Lemma \ref{l.preserverectangleimpliestorsion}; hence $g_{k,l}=g_{l,k}^{-1}$ for every $k,l\in \mathbb{N}$. Moreover, if $k>l$, then by the Markovian intersection axiom, the homeomorphism $\rho(g_{k,l})$ (resp. $\rho(g_{k,l})^{-1}$) acts as a contraction on the interval of $\mathcal{F}^u$-leaves (resp. $\mathcal{F}^s$-leaves) that intersect $R_k$ (resp. $R_l$). We deduce that, for every $k>l$, the map $\rho(g_{k,l})$ admits a unique fixed point $p_{k,l}$ inside $R_k\cap R_l$ (see  Lemma~\ref{l.twofixedpointsinthesamerectangle}), acts as a topological contraction on $\mathcal{F}^s(p_{k,l})$ and as a topological expansion on $\mathcal{F}^u(p_{k,l})$.

Choose now arbitrary orientations of the $\mathcal{F}^{s}$- and the $\mathcal{F}^{u}$-leaves of $R_0$. Our previous choices induce orientations on the $\mathcal{F}^{s}$- and the $\mathcal{F}^{u}$-leaves of $R_k$ for every $k\in\mathbb{N}$. After possibly passing to a subsequence of $(R_n)_{n\in\mathbb{N}}$, we may assume without loss of generality that $\rho(g_{k,l})$ preserves these orientations for every $k,l\in \mathbb{N}$. It follows that: 

\vspace{0.2cm}
\begin{fact}\label{fact.actiongkl}
    For every $k>l$ the homeomorphism $\rho(g_{k,l})$ preserves and acts as a topological contraction (resp. expansion) on each  $\mathcal{F}^{s}$-separatrix (resp. $\mathcal{F}^{u}$-separatrix) of $p_{k,l}$.
\end{fact} 

\begin{fact}\label{fact.boundarytoboundary}
    For every $k>l$ the homeomorphism $\rho(g_{k,l})$ sends the unique $\mathcal{F}^s$-boundary component of $R_k$ that is contained in $U$ to the unique $\mathcal{F}^s$-boundary component of $R_l$ that is contained in $U$.
\end{fact} 

\vspace{0.2cm}
\textbf{Claim 2.} For every $k,l\in \mathbb{N}$ with $k> l$ we have that $p_{k,l}\in U$.
\begin{proof}[Proof of Claim 2]
Assume by contradiction that there exist $k,l\in \mathbb{N}$ with $k>l$ such that $p_{k,l}\notin U$. 

First, notice that, since $l_-^s$ is not periodic, we have that $p_{k,l}\notin  l_-^s$. Next, by Fact \ref{fact.actiongkl}, $\rho(g_{k,l})$ acts as a topological expansion along the unique $\mathcal{F}^u$-separatrix of $p_{k,l}$ intersecting $U$ (the uniqueness of this separatrix follows from Proposition \ref{p.propertiesoffoliformarkovianactions}). Combining this with Fact 2 and the assumption that $p_{k,l}\notin U$, we obtain that 
\begin{equation}\label{eq.intersectlminus}
    \rho(g_{k,l})(l_-^s)\cap U\neq \emptyset
\end{equation} 
Furthermore, since $U$ is trivially bifoliated, we also have that $$\rho(g_{k,l})(l_-^s)\cap l^u_-\neq \emptyset \text{ and } \rho(g_{k,l})(l_-^s)\cap l^u_+\neq \emptyset$$ Hence, $l^u_-$ and $l^u_+$ intersect $\rho(g_{k,l})(U)$.

\begin{figure}[h!]
    \centering
    \includegraphics[width=0.8\linewidth]{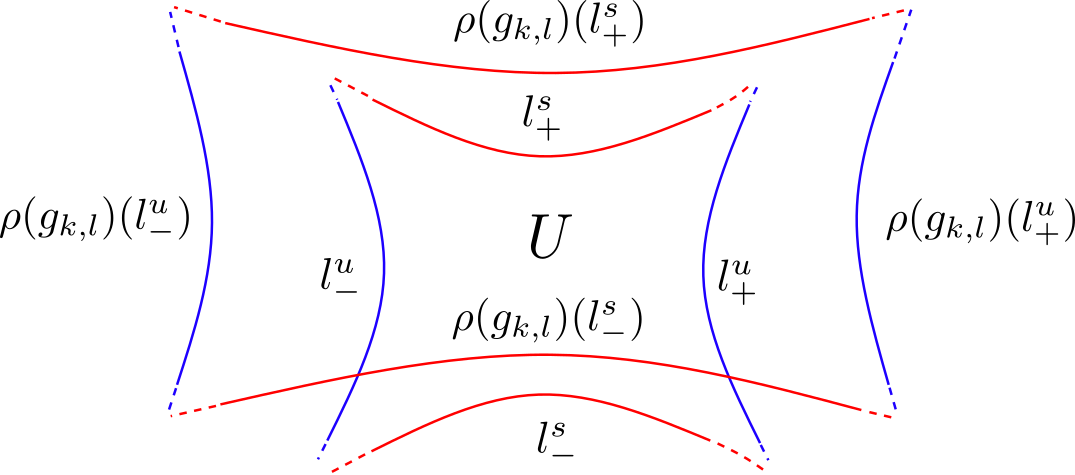}
    \caption{An impossible intersection pattern between totally ideal quadrilaterals arising from $p_{k,l}\notin U$.}
    \label{f.impossibleintersect}
\end{figure}

Assume now that $\rho(g_{k,l})(l_+^s)\cap U=\emptyset$ (see Figure \ref{f.impossibleintersect}). In this case, both $l_-^u$ and $l_+^s$ intersect $\rho(g_{k,l})(U)$ and thus, since $\rho(g_{k,l})(U)$ is trivially bifoliated, $l_-^u \cap l_+^s\neq \emptyset$, which contradicts our definition of $l_-^u$ and $l_+^s$. We deduce that 
\begin{equation}\label{eq.admissibleintersect}
    \rho(g_{k,l})(l_+^s)\cap U\neq\emptyset
\end{equation} 

Denote by $\mathcal{F}^u_U(p_{k,l})$ the unique $\mathcal{F}^u$-separatrix of $p_{k,l}$ that intersects $U$. By Fact \ref{fact.actiongkl}, $\rho(g_{k,l})$ preserves $\mathcal{F}^u_U(p_{k,l})$. Moreover, by (\ref{eq.intersectlminus}) and (\ref{eq.admissibleintersect}), the map $\rho(g_{k,l})$ acts as a contraction on $\mathcal{F}^u_U(p_{k,l})\cap U$ and therefore admits a fixed point in $\mathcal{F}^u_U(p_{k,l})\cap U$ that is distinct from $p_{k,l}$. This contradicts Lemma~\ref{l.periodicleavescarryuniqueperiodicpoint} and completes the proof of Claim~2.

\end{proof}

We remark that, by Claim 2 and the fact that $U$ is trivially bifoliated, the periodic point $p_{k,l}$ is regular for every $k,l\in\mathbb{N}$. Recall now that $g_{k,l}$ is the unique element in $G$ for which $\rho(g_{k,l})(R_{k})=R_l$. The previous fact implies that for every $k>l$ we have 
\begin{equation}\label{eq.relationgkl}
    g_{k,l}=g_{l+1,l} \cdot...  \cdot g_{k-1,k-2}\cdot g_{k,k-1}
\end{equation}

\begin{figure}[h!]
    \centering
    \includegraphics[width=0.5\linewidth]{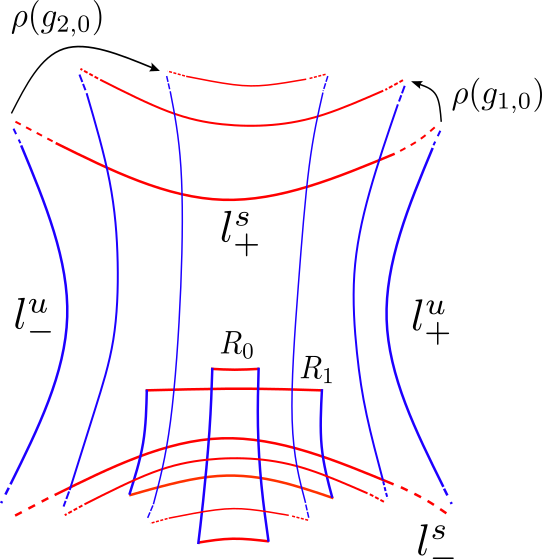}
    \caption{The intersection pattern of the totally ideal quadrilaterals $\rho(g_{k,0})(U)$.}
    \label{f.idealiteration}
\end{figure}

The above equation allows us to refine Claim~2 and to show that the sets $\rho(g_{k,0})(U)$ intersect ``Markovianly'' as in Figure \ref{f.idealiteration}.

\needspace{4\baselineskip}
\textbf{Claim 3.} For every $k,l\in \mathbb{N}$ with $k>l$ we have that: 

\begin{enumerate}
       \item $U\cap R_0 \subset \rho(g_{k,0})(U)$

    \vspace{0.1cm}

    \item $p_{k,l}\in \rho(g_{l,0})(U)\cap U$

    \vspace{0.1cm}

    \item $\rho(g_{k,0})(l_{\pm}^u)\cap \rho(g_{l,0})(U)\neq \emptyset$
    
    \vspace{0.1cm}
    
    \item $\rho(g_{k,0})(l_{\pm}^s)\cap \rho(g_{l,0})(U)= \emptyset$
\end{enumerate}
\begin{proof}[Proof of Claim 3]
  By Claim 2 and our construction of $p_{k,0}$, we have that $p_{k,0}\in R_{k}\cap R_0\cap U$ for every $k\in \mathbb{N}$. Also, by (\ref{eq.definitiongkl}), we have that $\rho(g_{k,0})(R_k)=R_0$. The previous facts, as well as Fact \ref{fact.actiongkl} imply that
  
  \begin{enumerate}[label=(\roman*)]
      \item $U\cap R_0 \subset \rho(g_{k,0})(R_k\cap U)\subset \rho(g_{k,0})(U)$, which completes the proof of Item~(1) of the claim. 

      \item $\rho(g_{k,0})(l_{\pm}^u)\cap U\neq \emptyset$ and $\rho(g_{k,0})(l_{\pm}^s)\cap U= \emptyset$, which completes the proof of Items~(3) and (4) of the claim in the case where $l=0$. 
  \end{enumerate}

More generally, assuming that Item~(2) of this claim holds, Items~(3) and~(4) follow easily from Fact ~\ref{fact.actiongkl} and the relation $g_{k,0}=g_{k,l}\cdot g_{l,0}$. It therefore suffices to prove that
\[
p_{k,l}\in \rho(g_{l,0})(U)\cap U
\]
for every $k>l$.

Assume that there exist $k',l'\in \mathbb{N}$ with $k'>l'$ such that $p_{k',l'}\notin  \rho(g_{l',0})(U)\cap U $. By Claim 2, the previous implies that $p_{k',l'}$ belongs in $U$, but not in $\rho(g_{l',0})(U)$. 

\begin{figure}[h!]
    \centering
    \includegraphics[width=0.5\linewidth]{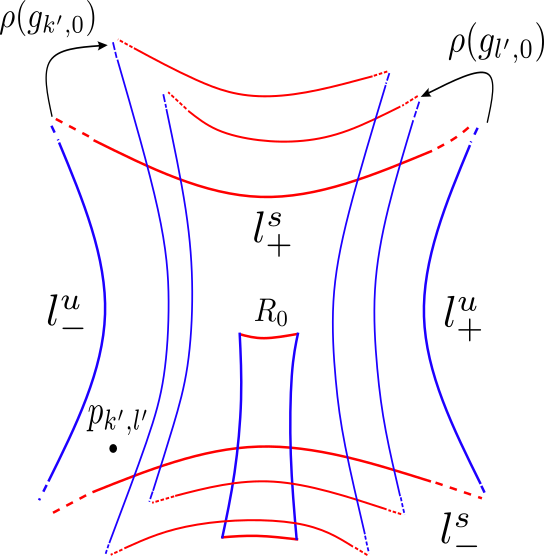}
    \caption{An impossible intersection pattern between totally ideal quadrilaterals arising from $p_{k',l'}\notin  \rho(g_{l',0})$.}
    \label{f.pklinsideul}
\end{figure}

Order the $\mathcal{F}^u$-leaves crossing $U$ from left to right so that $l^u_-$ is the leftmost leaf and $l^u_+$ is the rightmost one. Recall that by the first part of our proof, $\rho(g_{k',0})(l_{\pm}^u)\cap U\neq \emptyset$ and $\rho(g_{l',0})(l_{\pm}^u)\cap U\neq \emptyset$; hence, the leaves $$\rho(g_{k',0})(l_{-}^u), ~\rho(g_{k',0})(l_{+}^u), ~\rho(g_{l',0})(l_{-}^u), ~\rho(g_{l',0})(l_{+}^u)$$ all lie on the right of  $l^u_-$ and on the left of  $l^u_+$ (see Figure \ref{f.pklinsideul}). Without loss of generality, assume that $p_{k',l'}\in U$ lies to the right of $l^u_-$ and to the left of $\rho(g_{l',0})(l^u_-)$. Using the previous fact together with Fact 1 and the relation $g_{k',0}=g_{k',l'}\cdot g_{l',0}$, we get that the leaf $$\rho(g_{k',l'})(\rho(g_{l',0})(l^u_-))=\rho(g_{k',0})(l^u_-)$$ lies to the right of $l^u_-$ and to the left of $\rho(g_{l',0})(l^u_-)$. Furthermore, using Item (1) of this Claim, we have that $R_0\cap U \subset \rho(g_{k',0})(U)$. Combined with Fact \ref{fact.actiongkl}, this implies that the leaf $$\rho(g_{k',l'})(\rho(g_{l',0})(l^u_+))=\rho(g_{k',0})(l^u_+)$$ intersects $\rho(g_{l',0})(U)$ as in Figure \ref{f.pklinsideul}. Our previous arguments imply that both $\rho(g_{l',0})(l^s_-)$ and  $\rho(g_{l',0})(l^u_-)$ intersect $\rho(g_{k',0})(U)$. Since $\rho(g_{k',0})(U)$ is trivially bifoliated, we deduce that $$\rho(g_{l',0})(l^s_-)\cap \rho(g_{l',0})(l^u_-)\neq \emptyset$$ which is impossible as $l^s_-$ and $l^u_-$ are disjoint. 

\end{proof}

\noindent \textit{On the accumulation set of a sequence of totally ideal quadrilaterals}
\vspace{0.2cm}

We will now turn our attention to the set $V\subset \mathcal{P}$ on which $\rho(g_{k,0})(U)$ accumulates when $k\rightarrow +\infty$. Denote for the sake of simplicity $g_{k,0}$ by $g_k$. 

By Claim~3 and the fact that $U$ is trivially bifoliated, we get that the accumulation set $V$ of $(\rho(g_{k})(U))_{k\in \mathbb{N}}$ has non-empty interior, and this interior is trivially bifoliated. Concerning the boundary of $V$, we have the following two facts: 

\vspace{0.2cm}
\textbf{Claim 4.} Each of the sequences $(\rho(g_{k})(l^u_-))_{k\in \mathbb{N}}$ and $(\rho(g_{k})(l^u_+))_{k\in \mathbb{N}}$ accumulates on a face of a single $\mathcal{F}^u$-leaf. 

\vspace{0.2cm}
\textbf{Claim 5.} The accumulation set of the sequence $(\rho(g_{k})(l^s_+))_{k\in \mathbb{N}}$ contains the faces of at least two non-separated $\mathcal{F}^s$-leaves. In particular, the interior of $V$ is not a totally ideal quadrilateral. 

Before proving Claims~4 and~5, let us first complete the proof of the proposition assuming that these claims hold. Thanks to Claim 4, the sequence $(\rho(g_{k})(l^u_-))_{k\in \mathbb{N}}$ accumulates on a face of a single $\mathcal{F}^u$-leaf. Moreover, by Claim~5, we have that the accumulation set of $(\rho(g_{k})(l^s_+))_{k\in \mathbb{N}}$ contains the faces of at least two distinct non-separated $\mathcal{F}^s$-leaves. This implies that $V$ is the image of an embedding satisfying the hypotheses of Lemma \ref{l.embeddingweirdrectangle}, contradicting that lemma.

\vspace{0.3cm}
\noindent\textit{Proofs of Claims 4 and 5}

\begin{proof}[Proof of Claim 4]
We prove this claim for the sequence $(\rho(g_{k})(l^u_-))_{k\in \mathbb{N}}$; the case of $(\rho(g_{k})(l^u_+))_{k\in \mathbb{N}}$ follows from a similar argument. By Claim 3, $\rho(g_{k})(l^u_-)\cap U\neq \emptyset$ for every $k\in \mathbb{N}$; hence, the accumulation set of the sequence $(\rho(g_{k})(l^u_-))_{k\in \mathbb{N}}$ is non-empty and consists either of a face of a single $\mathcal{F}^u$-leaf or contains the faces of at least two non-separated $\mathcal{F}^u$-leaves, say $\mathcal{L}_1^{u},\mathcal{L}_2^{u}$. 

Assume that the latter case occurs. Without any loss of generality assume furthermore that $(\rho(g_{k})(l^u_-))_{k\in \mathbb{N}}$ accumulates on $\mathcal{L}_1^{u}$ and $\mathcal{L}_2^{u}$ from the right. By Claim~3, we deduce that the leaves of the sequence $(\rho(g_{k})(l^u_+))_{k\in\mathbb{N}}$ all lie to the right of both $\mathcal{L}_1^{u}$ and $\mathcal{L}_2^{u}$. Since $V$ and the sets $\rho(g_{k})(U)$ are trivially bifoliated, it is not difficult to see that both $\mathcal{L}_1^{u}$ and $\mathcal{L}_2^{u}$ are regular and non-separated from both the right and the left. This is impossible, because of Lemma \ref{l.nonseparationlemma} and the fact that both $\mathcal{L}_1^{u}$ and $\mathcal{L}_2^{u}$ are properly embedded in $\mathcal{P}$ (see Proposition \ref{p.propertiesoffoliformarkovianactions}).
\end{proof}

\begin{proof}[Proof of Claim 5]
We begin the proof of this claim by showing that the accumulation sets of the sequences $(\rho(g_{k})(l^s_+))_{k\in \mathbb{N}}$ and $(\rho(g_{k})(l^s_-))_{k\in \mathbb{N}}$ are both non-empty.

Indeed, recall that for every $k\in \mathbb{N}$ we have that $R_k\cap l_-^s\neq \emptyset$ and that $\rho(g_k)(R_k)=R_0$. It follows that $R_0\cap \rho(g_{k})(l^s_-)\neq \emptyset$ for every $k\in \mathbb{N}$, which proves that accumulation set of $(\rho(g_{k})(l^s_-))_{k\in \mathbb{N}}$ is indeed not empty. 

\begin{figure}[h!]
    \centering
    \includegraphics[width=0.45\linewidth]{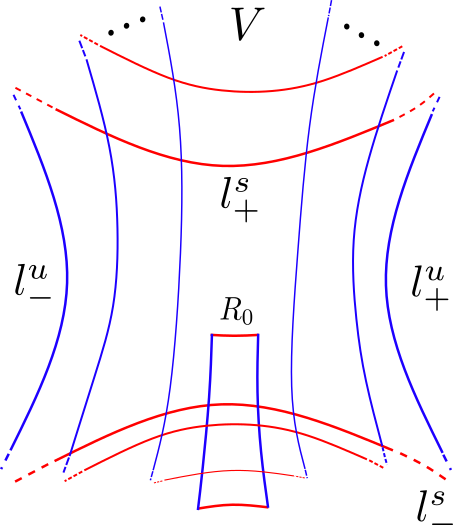}
    \caption{If $(\rho(g_{k})(l^s_+))_{k\in \mathbb{N}}$ exits every compact set in $\mathcal{P}$, then $V$ contains an infinite product region.}
    \label{f.Vcontainsinfiniteband}
\end{figure}

Next, assume by contradiction that $(\rho(g_{k})(l^s_+))_{k\in \mathbb{N}}$ exits every compact set in $\mathcal{P}$. In this case, thanks to Claim 4, the interior of $V$ is a trivially bifoliated region in $\mathcal{P}$ bounded only by the accumulation sets of $(\rho(g_{k})(l^u_+))_{k\in \mathbb{N}}$, $(\rho(g_{k})(l^u_-))_{k\in \mathbb{N}}$ and $(\rho(g_{k})(l^s_-))_{k\in \mathbb{N}}$ (see Figure \ref{f.Vcontainsinfiniteband}). Thanks to this fact, there exists an infinite product region in $V$. By Lemma \ref{l.noinfiniteproductregion}, this implies that the bifoliation $(\mathcal{F}^s,\mathcal{F}^u)$ is trivial, which contradicts the existence of the totally ideal quadrilateral $U$.

By our previous arguments, the accumulation set of $(\rho(g_{k})(l^s_+))_{k\in \mathbb{N}}$ is non-empty; hence, it consists either of a face of a single $\mathcal{F}^s$-leaf or of the union of faces of at least two non-separated $\mathcal{F}^s$-leaves. Assume by contradiction that $(\rho(g_{k})(l^s_+))_{k\in \mathbb{N}}$ accumulates on a face of a single $\mathcal{F}^s$-leaf, which we will denote by $\mathcal{L}_+^{s}$. 

Notice that, by Claim 3 and since $V$ is trivially bifoliated we have that
\begin{equation}\label{eq.intersectionLs}
 \forall y\in \inte{V} ~~   \mathcal{F}^u(y)\cap \mathcal{L}_+^{s}\neq \emptyset 
\end{equation}

Recall now that, for every $k\in\mathbb{N}^*$, the map $\rho(g_{k})$ admits a fixed point $p_{k,0}$ in $R_k\cap R_0$, acts as an expansion on each $\mathcal{F}^u$-separatrix of $p_{k,0}$, and as a contraction on each $\mathcal{F}^s$-separatrix of $p_{k,0}$. We also recall that $g_k\neq g_l$ whenever $k\neq l$. Thanks to the previous facts and also to Lemmas \ref{l.npredecessor} and \ref{l.preserverectangleimpliestorsion}, we have that $(\rho(g_k)(R_0))_{k\in\mathbb{N}^*}$ forms a sequence of mutually distinct predecessors of $R_0$ of some generation. Since for every $N\in \mathbb{N}^*$ the predecessors of $R_0$ of generation at most $N$ are finitely many (see Lemma \ref{l.existenceofpredecessors}), after possibly considering a subsequence, assume without any loss of generality that $\rho(g_{k+1})(R_0)$ is a predecessor of $R_0$ of greater generation than $\rho(g_{k})(R_0)$. 

Thanks to our previous assumption and to the expansivity axiom, as $k$ increases, the rectangle $\rho(g_k)(R_0)\cap R_0$ becomes ``thinner'' in the $\mathcal{F}^s$-direction. After possibly passing to a subsequence, assume that the rectangles $\rho(g_k)(R_0)\cap R_0$ converge to an $\mathcal{F}^u$-segment in $R_0$ (see Figure \ref{f.predecessorsnonperi}). Let $x_{\infty}$ be a point in this segment, and let $D$ be the closure of the connected component of $\mathcal{P}-\mathcal{F}^u(x_{\infty})$ containing $\rho(g_k)(R_0)$ for infinitely many $k\in\mathbb{N}^*$. Passing to a further subsequence if necessary, we may assume that
$$\rho(g_k)(R_0)\subset D$$
for every $k\in\mathbb{N}^*$.

\vspace{0.3cm}
\textbf{Case 1: $\mathcal{F}^u(x_{\infty})$ is not a periodic $\mathcal{F}^u$-leaf}

Thanks to (\ref{eq.intersectionLs}), we have that $\mathcal{F}^u(x_{\infty})\cap \mathcal{L}_+^{s}\neq \emptyset$. Furthermore, since $\mathcal{F}^u(x_{\infty})$ is non-periodic, Lemma~\ref{l.longrectanglesregularleavesnoperiodicpoints} implies that there exists $r\in \mathcal{R}$ such that $r\cap \inte{D}\neq \emptyset $, $r\cap \mathcal{L}_+^{s}\neq \emptyset$ and also such that $x_{\infty}\in r$. 

Since $\rho(g_k)(R_0)\cap R_0$ converges to $\mathcal{F}^u(x_{\infty})\cap R_0$ when $k\rightarrow \infty$, there exist infinitely $k\in \mathbb{N}^*$ for which $$\rho(g_k)(\inte{R_0})\cap \inte{r}\neq \emptyset $$

On the one hand, thanks to the previous fact, to Remark \ref{r.precedentsuivant} and also to the fact that $\rho(g_k)(R_0)$ is a predecessor of $R_0$ whose generation increases with $k$, we get that there exists $K\in \mathbb{N}^*$ for which $\rho(g_K)(R_0)$ is a predecessor of some generation of $r$. In particular, there exists $K\in \mathbb{N}^* $ for which $$\rho(g_K)(R_0)\cap \mathcal{L}_+^{s}\neq \emptyset$$ On the other hand, since $\partial^uR_0\cap l^s_+=\emptyset$, we have that $\rho(g_l)(\partial^uR_0)\cap  \rho(g_l)(l^s_+)=\emptyset$ for every $l\in \mathbb{N}^*$; hence, by Claim 3 there can not exist $k\in \mathbb{N}^*$ such that $\rho(g_k)(R_0)\cap \mathcal{L}_+^{s}\neq \emptyset$. This yields a contradiction.

\begin{figure}[h!]
    \centering
    \includegraphics[width=0.45\linewidth]{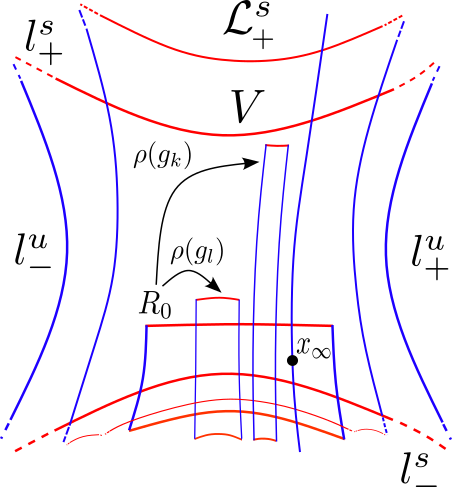}
    \caption{The accumulation of the $\rho(g_n)(R_0)$ on $\mathcal{F}^u(x_{\infty})$.}
    \label{f.predecessorsnonperi}
\end{figure}

\needspace{5\baselineskip}

\textbf{Case 2: $\mathcal{F}^u(x_{\infty})$ is a periodic $\mathcal{F}^u$-leaf}

Denote by $p_\infty$ the unique periodic point for $\rho$ inside $\mathcal{F}^u(x_\infty)$. If there exists $r\in \mathcal{R}$ such that $r\cap \inte{D}\neq \emptyset $, $r\cap \mathcal{L}_+^{s}\neq \emptyset$ and also such that $x_{\infty}\in r$, then by the same argument used in Case 1 we obtain a contradiction. Let us now assume that no such rectangle exists. The previous fact combined with (\ref{eq.intersectionLs}) and Lemmas  \ref{l.longrectanglesregularleaves} and \ref{l.longrectanglessingularleaves} implies the following: 

\begin{itemize}
    \item $p_\infty \in \inte{V}$. More specifically, $p_\infty$ belongs in the $\mathcal{F}^u$-segment going from $x_\infty$ to $\mathcal{L}^s_+$ (see Figure \ref{f.predecesseurperiodique}). Therefore, since the interior of $V$ is trivially bifoliated, $p_\infty$ is non-singular.
    %\item There does not exist a rectangle in $\mathcal{R}$ containing a neighborhood of $p_\infty$ inside $D$. Therefore, $p_\infty\notin \inte{R_0}$.
    \item There exist $r_1,r_2\in \mathcal{R}$ such that $r_1\cap \inte{D}\neq \emptyset $, $r_2\cap \inte{D}\neq \emptyset $, $p_\infty,x_\infty \in r_1$, $p_\infty\in r_2$ and $r_2\cap \mathcal{L}_+^{s}\neq \emptyset$ (see Figure \ref{f.predecesseurperiodique}).
\end{itemize}

\begin{figure}[h!]
    \centering
    \includegraphics[width=0.4\linewidth]{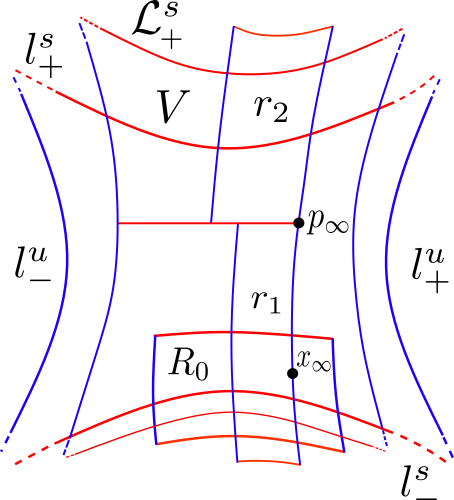}
    \caption{The case where $\mathcal{F}^u(x_{\infty})$ is periodic.}
    \label{f.predecesseurperiodique}
\end{figure}

By construction, $\rho(g_k)(R_0)\cap R_0$ converges to $\mathcal{F}^u(x_\infty)\cap R_0$ when $k\rightarrow +\infty$. It follows that there exist infinitely many $k\in\mathbb{N}^*$ for which $\rho(g_k)(R_0)$ intersects the interior of $r_1$. If there exist infinitely many $k\in\mathbb{N}^*$ for which $\rho(g_k)(R_0)$ intersects the interiors of both $r_1$ and $r_2$, then we can apply an argument similar to the one used in Case 1 in order to produce $K\in \mathbb{N}$ for which $\rho(g_K)(R_0)\cap \mathcal{L}^s_+\neq \emptyset$ and obtain a contradiction. After possibly considering a subsequence, let us therefore assume that for every $k\in\mathbb{N}^*$ 
\begin{equation}\label{eq.intersectionr1r2}
    \rho(g_k)(R_0)\cap \inte{r_1}\neq \emptyset \text{ and } \rho(g_k)(R_0)\cap \inte{r_2}= \emptyset
\end{equation}

Denote by $\mathcal{F}^s_-(p_\infty)$ the $\mathcal{F}^s$-separatrix of $p_\infty$ that intersects $D$ non-trivially. Thanks to (\ref{eq.intersectionr1r2}), to Remark \ref{r.precedentsuivant} and to the fact that  $\rho(g_k)(R_0)$ is a predecessor of $R_0$ whose generation increases with $k$, we get that for $k$ sufficiently big $\rho(g_k)(R_0)$ is a predecessor of $r_1$ of some generation and also that  

\begin{equation}\label{eq.takesinseparatrix}
  \rho(g_k)(\partial^uR_0\cap U)\subset \mathcal{F}^s_-(p_\infty)  
\end{equation}

Take $K>L$ sufficiently big so that both $\rho(g_K)(\partial^uR_0\cap U)$ and $\rho(g_L)(\partial^uR_0\cap U)$ are contained in $\mathcal{F}^s_-(p_\infty)$. Thanks to the previous fact, we have that $\mathcal{F}^s(p_\infty)$ intersects $\rho(g_L)(U)$. Moreover, by Claim~3,
$$U\cap R_0\subset \rho(g_L)(U);$$
hence, $\mathcal{F}^u(x_\infty)$ also intersects $\rho(g_L)(U)$. The previous facts, together with the fact that $\rho(g_L)(U)$ is trivially bifoliated, imply that $p_\infty$, the unique point of intersection of $\mathcal{F}^u(x_\infty)$ and $\mathcal{F}^s(p_\infty)$, lies inside $\rho(g_L)(U)$.

Recall now that $g_K=g_{K,L}\cdot g_L$. Thanks to our choice of $K,L$ we get that $\rho(g_{K,L})$ preserves $\mathcal{F}^s_-(p_\infty) $ and thus fixes $p_{\infty}$, the unique periodic point in $\mathcal{F}^s_-(p_\infty) $. On the other hand, our definition of $g_{K,L}$ and Claim 3 imply that $\rho(g_{K,L})$ fixes a point $p_{K,L}$ inside $ R_K\cap R_L\cap \rho(g_L)(U)\cap U$. 

A priori, $p_{K,L}$ may coincide with $p_{\infty}$. Nevertheless, we will show that, by changing our choice of $K$ and $L$, we may assume that $p_{K,L}\neq p_{\infty}$. Indeed, assume that there exists $L_0\in \mathbb{N}$ such that $p_{K,L}=p_{\infty}$ for every $K>L>L_0$. Then, since $p_{K,L}\in R_K\cap R_L$ we have that 
$$p_{\infty}\in \underset{k>L_0}{\cap}R_k$$
Moreover, by the expansivity axiom and our definition of the sequence $(R_n)_{n\in\mathbb{N}}$, 
$$\underset{k>L_0}{\cap}R_k\subset l^s_-$$
which would imply that $p_{\infty}\in l^s_-$ and thus that $l^s_-$ is periodic, contradicting the first part of our proof. We can therefore assume without any loss of generality that $p_{K,L}\neq p_{\infty}$.

Finally, since both $p_{K,L}$ and $ p_{\infty}$ are fixed by $\rho(g_{K,L})$ and are contained in the trivially bifoliated region $\rho(g_L)(U)$, we have that $\rho(g_{K,L})$ also fixes the unique point of intersection of $\mathcal{F}^u(p_{\infty})$ and $\mathcal{F}^s(p_{K,L})$ inside $\rho(g_L)(U)$. This contradicts Lemma~\ref{l.periodicleavescarryuniqueperiodicpoint} and finishes the proof of Claim~5.

\end{proof}

\end{proof}

\end{document}